\documentclass[journal,twoside,web]{ieeecolor}
\usepackage{generic}
\usepackage{cite}
\usepackage{amsmath,amssymb,amsfonts}
\usepackage{algorithmic}
\usepackage{graphicx}
\usepackage{subcaption}
\usepackage{algorithm,algorithmic}
\usepackage{hyperref}
\hypersetup{hidelinks=true}
\usepackage{textcomp}

\newtheorem{theorem}{Theorem}[section]
\newtheorem{proposition}[theorem]{Proposition}
\newtheorem{lemma}[theorem]{Lemma}

\newtheorem{remark}[theorem]{Remark}
\newtheorem{definition}[theorem]{Definition}
\newtheorem{example}[theorem]{Example}
\newtheorem{problem}{Problem}

\newtheorem{assumption}{Assumption}

\def\BibTeX{{\rm B\kern-.05em{\sc i\kern-.025em b}\kern-.08em
    T\kern-.1667em\lower.7ex\hbox{E}\kern-.125emX}}
\begin{document}
\title{Timed Recovery Safe Gradient Flows for Constrained Nonlinear Programming}
\author{Tianxi Luo, Yue Wu, Fei Chen, ~\IEEEmembership{Senior Member,~IEEE}, and Hongsheng Qi
\thanks{This work was supported in part by the National Natural Science Foundation of China under Grant 62573240 and Grant 61973061; and in part by the TCL Technology Innovation Fund of TCL Technology Group Corporation. (Corresponding author: Fei Chen.)}
\thanks{Tianxi Luo and Fei Chen are with the College of Artificial Intelligence, Nankai University, Tianjin 300350, China (e-mail: 2120250490@mail.nankai.edu.cn; fchen@nankai.edu.cn).}
\thanks{Yue Wu and Hongsheng Qi are with the State Key Laboratory of Mathematical Sciences, Academy of Mathematics and Systems Science, Chinese Academy of Sciences, Beijing 100190, China, and with the School of Mathematical Sciences, University of Chinese Academy of Sciences, Beijing 100049, China (e-mail:wuyue24@amss.ac.cn; qihongsh@amss.ac.cn)
    }}

\maketitle

\begin{abstract}
This paper studies how to relax pointwise feasibility in continuous-time dynamics for inequality-constrained nonlinear programming while retaining a quantitative feasible-recovery guarantee. Existing safe gradient flows enforce forward invariance of the feasible set, which yields an anytime-feasible dynamics but can be overly restrictive for optimization because it suppresses transient infeasible motion that may be essential for effective descent. To address this limitation, we propose a timed recovery safe gradient flow that replaces pointwise feasibility with fixed-horizon feasible recovery. The proposed construction augments the state with a timer that records the elapsed infeasible duration and computes a minimum-norm correction to the nominal gradient flow through a pointwise convex quadratic program. We show that the resulting augmented safe set is forward invariant and that every trajectory initialized from the feasible set returns to it within every time window of length \(\tau\). We further show that, in a strict interior inactive regime, the closed-loop dynamics reduces exactly to the nominal gradient flow and recovers the corresponding local convergence behavior. Numerical examples illustrate how the recovery horizon shapes the closed-loop behavior and corroborate the theoretical analysis.
\end{abstract}

\begin{IEEEkeywords}
Constrained nonlinear programming, control barrier functions, continuous-time optimization, gradient flows, recurrent safety. 
\end{IEEEkeywords}

\section{Introduction}
Continuous-time formulations provide a natural language for viewing optimization problems as dynamical processes rather than purely numerical procedures. This viewpoint has a long history in optimization and control, and has developed into a systematic framework for analyzing descent, stability, and convergence from a dynamical-systems perspective \cite{1958Studies,BROCKETT199179,503147,HAUSWIRTH2024100941,hauswirth2021projected}. It is especially natural for constrained nonlinear programming, where one must not only drive the state toward lower objective values, but also shape the evolution so that it remains compatible with the geometry of the feasible set. In this sense, constrained optimization acquires an intrinsically control-theoretic character: the objective prescribes the desired motion, while the constraints determine how that motion can be realized. This viewpoint underlies a broad spectrum of continuous-time methods for constrained optimization, including differential-geometric feasible flows, projected dynamical systems, nonsmooth constraint-aware first-order methods, and primal-dual optimization dynamics \cite{TANABE1979707, nagurney2012projected, FEIJER20101974, 15M1026924}.

Among recent control-theoretic approaches, a particularly influential formulation is the safe gradient flow proposed in \cite{10224270}. There, the feasible set is treated as a safe set, and a feedback law based on control barrier function(CBF) is synthesized so that the resulting dynamics preserves feasibility through forward invariance while driving the state toward critical points of the constrained problem. This perspective is appealing because it places constrained nonlinear programming directly within a feedback-design framework and endows the flow with an anytime-feasibility interpretation. More broadly, it has stimulated a growing line of work on control-theoretic continuous-time optimization, including feedback-based multiplier dynamics for constrained optimization \cite{10994413}, safe extensions to bilevel optimization \cite{11107448}, safety-preserving methods for time-varying nonlinear programming with equality constraints \cite{11208701}, and CBF-based constrained gradient flows with finite-time and fixed-time convergence guarantees \cite{11222659}. At the same time, recent advances in feedback-based online optimization and constrained first-order methods further highlight how continuous-time dynamics, primal-dual couplings, and augmented-Lagrangian ideas can be used to connect optimization algorithms with closed-loop system design \cite{colombino2019online, de2023constrained, zhu2024first}. Taken together, these developments point to a common theme: constrained optimization can be studied as the synthesis of dynamics in which objective-driven descent and constraint handling are inseparably intertwined. 

However, when the primary objective is optimization rather than safety at every instant, strict pointwise feasibility can be too rigid. Near the feasible boundary, the nominal descent direction may naturally point outward, so any feasible-by-design dynamics must override that direction exactly where the objective most strongly favors leaving the set \cite{ALI199857,TANABE1979707, nagurney2012projected, 10224270}. Pure optimization methods avoid this rigidity, but they do so in ways that still do not provide the semantics sought here. Interior-point and barrier-based dynamics preserve strict interiority, yet for finite barrier parameters they evolve on a barrier-regularized surrogate rather than directly on the original constrained problem \cite{Fiacco1968NonlinearPS, 2006On, fazlyab2017prediction}. Penalty, multiplier, and augmented-Lagrangian methods do allow temporary infeasibility, but violations are handled indirectly through penalty terms or multiplier updates, so the dynamics is governed by a modified objective or augmented Lagrangian rather than by an explicit recovery requirement on the state trajectory \cite{dhingra2018proximal, de2023constrained, zhu2024first, tang2022running, jia2023augmented}. Likewise, recent first-order constrained methods can permit infeasible iterates and local constraint approximations, but they are generally designed for asymptotic feasibility or stationarity rather than for certifying that every infeasible excursion must end within a prescribed time horizon \cite{muehlebach2021constraints}. What remains missing is a continuous-time optimization dynamics that allows temporary infeasibility without replacing the problem by a barrier or penalty surrogate, while still guaranteeing quantitative bounded-time return to feasibility.

This mismatch becomes especially sharp when the feasible set is disconnected. In that case, pointwise feasibility is no longer merely a safety requirement; it becomes a reachability restriction. Any forward-invariant dynamics initialized in one feasible component is confined to that component for all future time, so the set of reachable local optima is already topologically determined by the initial condition. This is not a pathological corner case. In general nonlinear programming, Lee and Chiang showed that disconnected feasible components naturally give rise to multiple local optima and explicitly motivate dynamical mechanisms capable of moving across components \cite{1304911}. Similar phenomena also arise in important problems such as AC optimal power flow, where disconnected feasible regions, multiple local optima, and the decisive role of feasible-set geometry have all been documented \cite{Gomez1982, bukhsh2013local,zhang2012geometry,lavaei2013geometry}. From the perspective of the present work, this is precisely the regime in which strict feasibility preservation becomes too rigid for optimization. 

These observations bring the unresolved issue into sharp focus. Continuous-time constrained optimization already has two well-understood limiting semantics: unconstrained descent, which preserves optimization-relevant motion at the price of feasibility, and forward-invariant safe optimization, which preserves feasibility at the price of suppressing such motion. What is still missing is the trajectory-level constraint semantics between these two limits, one that permits temporary infeasibility while still regulating its duration and recovery. The central question is therefore this: can instantaneous feasibility be relaxed in a way that still enforces a rigorous and quantitative law of return to the feasible set?

Recurrence provides a natural framework for expressing this intermediate requirement, but its use in optimization is not merely a direct import of existing recurrent-safety ideas. Bounded-return concepts have appeared in region-of-attraction learning \cite{9993280}, recurrence-based stability analysis \cite{10383373}, constrained predictive control \cite{SIBAI2026101649}, and persistent monitoring with recharging \cite{2025Multirobot}. Within barrier-function theory, recurrent barrier conditions formalize this viewpoint by replacing forward invariance with bounded-time return to a target superlevel set \cite{10735269}. The challenge in the present setting is different and more delicate. The trajectory is not generated by an arbitrary controlled system whose only task is to return to a prescribed target set; it is driven by the descent direction of the original objective. Thus, the recovery mechanism must regulate infeasible excursions without destroying the optimization motion that made such excursions useful in the first place. In particular, it must allow the trajectory to leave the feasible set $C_g$ when doing so benefits descent, yet enforce a quantitative return law before the violation persists beyond the prescribed horizon. To meet this challenge, we use recurrence not as a property to be certified after the dynamics are given, but as a design principle for constructing the dynamics themselves. We encode fixed-horizon feasible recovery through a timed recovery safe gradient flow, which augments the nominal gradient dynamics with a timer-based recovery constraint. At each state, a pointwise convex quadratic program computes the minimum-norm correction required to satisfy this constraint, thereby enforcing bounded-time return to $C_g$ while perturbing the original gradient flow only when necessary.

This paper makes three contributions.
First, we introduce fixed-horizon feasible recovery as a trajectory-level constraint semantics for continuous-time nonlinear programming. Unlike forward-invariant safe gradient flows, which require $x(t)\in \mathcal C_g$ for all $t$, the proposed semantics allows transient infeasibility while ruling out infeasible intervals longer than a prescribed horizon $\tau$. Second, we construct a timer-regulated recurrent feasibility filter for nominal gradient flows. The filter augments the state with a recovery-budget variable and computes, through a pointwise convex quadratic program, the minimum-norm correction required to satisfy a recurrent barrier condition. Thus the nominal gradient direction is preserved whenever the recovery constraint is inactive. Third, we characterize the behavior induced by the recovery horizon. We prove forward invariance of the augmented budget set and
$\tau$-recurrent return to the original feasible set; show that the feasible boundary becomes permeable while the dynamics becomes increasingly restorative as the recovery budget is depleted; and establish two limiting regimes: the zero-horizon limit recovers a CBF-type safe gradient flow, whereas the infinite-horizon limit recovers the nominal gradient flow. We also identify a strict interior inactive regime in which the closed loop inherits the local convergence behavior of the nominal gradient flow.


The paper is organized as follows. Section~\ref{sec:preliminaries_problem} presents the preliminaries and problem formulation. Section~\ref{sec:rcbf} develops the timer-regulated recurrent feasibility filter. Section~\ref{sec:optimization_dynamics} constructs the fixed-$\tau$
timed recovery safe gradient flow and establishes its recurrent safety properties. Section~\ref{sec:system_behavior} analyzes the resulting qualitative behavior, compatibility, and horizon-dependent tradeoffs. Section~\ref{sec: simulations} reports numerical examples. 

\textbf{\textit{Notation: }}
$\operatorname{dist}(\mathbf{z}, \mathcal{A}) \triangleq \inf_{\mathbf{y} \in \mathcal{A}} \| \mathbf{z} - \mathbf{y} \|$ denotes the distance from a point $\mathbf{z} \in \mathcal{X}$ to a set $\mathcal{A} \subset \mathcal{X}$. For any scalar $a\in\mathbb{R}$, $[a]_+ \triangleq \max\{a,0\}$ denotes the positive part of $a$. For any set \(S\), \(\partial S\) and \(\operatorname{int}(S)\) denote the boundary and interior of \(S\), respectively. A continuous function \(\alpha:\mathbb R\to\mathbb R\) is called an extended class \(\mathcal K\) function if it is strictly increasing and satisfies \(\alpha(0)=0\).

\section{Preliminaries and Problem Formulation}

\label{sec:preliminaries_problem}
This section first presents the fundamental concepts used throughout the paper, and then formulates the fixed-horizon timed recovery design problem.

\subsection{Preliminaries}

Throughout this subsection, we consider a locally Lipschitz dynamical system
\begin{equation}\label{eq:general_dynamics}
    \dot x = F(x),
\end{equation}
defined on a state space \(\mathcal X\subseteq \mathbb R^N\). For any initial condition \(x_0\in\mathcal X\), we denote by \(\phi(t,x_0)\) the corresponding solution at time \(t\), whenever it is well defined.

\begin{definition}[Forward invariance]\cite{10224270}
Let \(S\subseteq \mathcal X\). The set \(S\) is said to be \emph{forward invariant} if every trajectory starting from \(S\) remains in \(S\) for all future time, that is,
\[
x_0\in S
\quad\Longrightarrow\quad
\phi(t,x_0)\in S,
\quad \forall t\ge 0.
\]
\end{definition}

\begin{definition}[Recurrent and \(\tau\)-recurrent sets]\cite{10735269}
Let \(S\subseteq \mathcal X\). The set \(S\) is called \emph{recurrent} if, for every \(x\in S\) and every \(t\ge 0\), there exists \(t'>t\) such that
\[
\phi(t',x)\in S.
\]
Given a constant \(\tau>0\), the set \(S\) is called \(\tau\)-\emph{recurrent} if, for every \(x\in S\) and every \(t\ge 0\), there exists \(t'\in(t,t+\tau]\) such that $\phi(t',x)\in S$. 
\end{definition}

The definition of $\tau$-recurrence requires that every trajectory starting from $S$ returns to $S$ 
within any time window of length $\tau$. This allows temporary infeasibility, while ensuring that every infeasible excursion has a uniformly bounded duration.

\begin{definition}[Karush--Kuhn--Tucker point]\cite{nocedal2006numerical}
A feasible point \(x^*\in\mathbb R^n\) is called a Karush--Kuhn--Tucker point of \eqref{eq:nlp} if there exists a multiplier vector \(p^*\in\mathbb R^m\) such that
\begin{subequations}\label{eq:kkt}
\begin{align}
\nabla f(x^*)+\frac{\partial g(x^*)}{\partial x}^{\top}p^* &= 0,\\
g(x^*) &\le 0,\\
p^* &\ge 0,\\
(p^*)^\top g(x^*) &= 0.
\end{align}
\end{subequations}
\end{definition}
We denote by
\[
X_{\mathrm{KKT}}
\triangleq
\{x\in\mathcal C_g \mid x \text{ is a Karush--Kuhn--Tucker point of \eqref{eq:nlp}}\}
\]
the set of all KKT points of \eqref{eq:nlp}.

\begin{definition}[Upper right Dini derivative]\cite{rockafellar1998variational}
Let \(v:[0,\infty)\to\mathbb R\) be a continuous function. The upper right Dini derivative of \(v\) at time \(t\ge 0\) is defined by
\[
D^{+}v(t)\triangleq \limsup_{h\downarrow 0}\frac{v(t+h)-v(t)}{h}.
\]
If \(v\) is differentiable at \(t\), then \(D^{+}v(t)=\dot v(t)\).
\end{definition}

\begin{definition}[Periodic orbit] \cite{mestres2026control}
Consider the system $\dot{x} = F(x)$
with flow \(\phi(\cdot,x_0)\). A nonempty set \(\Gamma\subseteq \mathcal X\) is called a \emph{periodic orbit} if there exist \(x_0\in\mathcal X\) and a minimal period \(T>0\) such that
\[
\phi(t+T,x_0)=\phi(t,x_0),
 \forall t\ge 0, \phi(t,x_0)\neq x_0, \forall t\in(0,T),
\]
and
\[
\Gamma=\{\phi(t,x_0)\mid t\in[0,T)\}.
\]
\end{definition}

\subsection{Problem Formulation}

In this paper, we consider the inequality-constrained nonlinear program
\begin{equation}\label{eq:nlp}
\begin{aligned}
\min_{x\in\mathbb R^n}\quad & f(x)\\
\text{s.t.}\quad & g(x)\le 0,
\end{aligned}
\end{equation}
where \(f:\mathbb R^n\to\mathbb R\) and \(g:\mathbb R^n\to\mathbb R^m\) are continuously differentiable. Its feasible set is
\begin{equation}\label{eq:feasible_set}
\mathcal C_g \triangleq \{x\in\mathbb R^n\mid g(x)\le 0\}.
\end{equation}

\begin{problem}[Fixed-horizon timed recovery safe gradient flow] Consider the constrained optimization problem \eqref{eq:nlp} and a prescribed horizon $\tau > 0$. Design a continuous-time dynamical system of the form
\begin{equation}
\dot{x} = -\nabla f(x) + u(x,\eta),
\end{equation}
where $u(x,\eta) \in \mathbb{R}^n$ is a pointwise computable correction term and $\eta$ is an auxiliary state, such that the following properties hold:
\begin{enumerate}
\item (\textbf{$\tau$-recurrent feasibility}) 
For every trajectory with $x(0) \in \mathcal C_g$, the state returns to $\mathcal C_g$ within every time interval of length $\tau$.

\item (\textbf{Constraint-enforcing correction}) 
The correction $u(x,\eta)$ is chosen such that the closed-loop velocity satisfies a pointwise constraint ensuring $\tau$-recurrent feasibility. 
Moreover, the correction is inactive whenever the nominal gradient direction already satisfies this constraint, i.e.,
\begin{equation}
u(x,\eta) = 0 \quad \text{whenever } -\nabla f(x) \text{ is admissible}.
\end{equation}

\item (\textbf{Minimal modification principle}) 
Among all corrections that enforce the constraint, $u(x,\eta)$ is selected to minimize the instantaneous deviation from the nominal gradient flow.

\item (\textbf{Consistency with optimization structure}) 
The resulting dynamics coincides with the nominal gradient flow on a subset of the feasible region, and its equilibrium points satisfy necessary optimality conditions consistent with the KKT structure of problem \eqref{eq:nlp}.
\end{enumerate}
\end{problem}

\section{Timed Recovery Control Barrier Function Framework}
\label{sec:rcbf}
In this section, we develop a timed recovery control barrier mechanism for systems subject to vector inequality constraints. 
This mechanism serves as the technical backbone for the subsequent sections, where it will be further integrated into the timed recovery safe gradient flow framework.
For the sake of clarity, this section considers the dynamics of general nonlinear systems \eqref{eq:general_dynamics}.

\subsection{Timed Recovery Control Barrier Function Design}
Decompose the vector-valued constraint $g(x)\le 0$ into its scalar components $g_i(x)\le 0$, for $1\le i\le m$.
Then, the inequality feasible set defined in \eqref{eq:feasible_set} becomes
\begin{equation} 
\mathcal C_g = \{x \in \mathbb{R}^n \mid g_i(x)\le 0,\ i=1,\dots,m\}.
\end{equation}
Define the aggregated inequality violation function by
\begin{equation}\label{eq:gbar_def}
\bar g(x)\triangleq \max_{1\le i\le m} g_i(x).
\end{equation}
Then $x\in \mathcal C_g$ if and only if $\bar g(x)\le 0$. For each $x\in\mathbb{R}^n$, let
\begin{equation}\label{eq:Imax_def}
\mathcal{I}_{\max}(x)\triangleq \{i\in\{1,\dots,m\}\mid g_i(x)=\bar g(x)\}
\end{equation}
denote the set of active maximizers.

To measure the duration of collective inequality violation, we introduce a scalar timer $\eta\in\mathbb{R}$ with dynamics
\begin{equation}\label{eq:eta_dyn}
\dot{\eta}=\rho(x,\eta)\triangleq
\begin{cases}
1, & \bar g(x)>0,\\
-\lambda \eta, & \bar g(x)\le 0,
\end{cases}
\quad \eta(0)=0,
\end{equation}
where \(\lambda>0\) is fixed. The timer \(\eta\) converts transient infeasibility into a continuous recovery state: it increases at unit speed whenever \(\bar g(x)>0\), so its increment over an infeasible interval equals the duration of that interval. Once \(\bar g(x)\le 0\), \(\eta\) decays exponentially rather than being reset discontinuously, with \(\lambda\) determining the rate at which past infeasibility is forgotten. In the limiting regime \(\lambda\to+\infty\), this exponential decay approaches an instantaneous timer reset after feasibility is recovered, while preserving a continuous-time formulation for every finite \(\lambda\). The following lemma records the non-negativity of \(\eta\), which will be used in the subsequent recovery-budget analysis. 

\begin{lemma}[Nonnegativity of the timer]\label{lem:eta_nonnegative}
Consider the timer defined in \eqref{eq:eta_dyn}, then $\eta(t)\ge 0$ for all $t\ge 0$.
\end{lemma}
\begin{proof}
Let
\[
h(\eta)\triangleq \eta,
\quad
\mathcal K\triangleq \{\eta\in\mathbb R\mid h(\eta)\ge 0\}=[0,\infty).
\]
We show that $\mathcal K$ is forward invariant.

On the boundary $\partial \mathcal K=\{0\}$, one has $h(\eta)=0$. Evaluating the timer dynamics \eqref{eq:eta_dyn} at $\eta=0$, we obtain
\[
\dot h(\eta)=\dot\eta=
\begin{cases}
1, & \bar g(x)>0,\\
-\lambda \eta, & \bar g(x)\le 0,
\end{cases}
\]
and hence
\[
\dot h(0)=
\begin{cases}
1, & \bar g(x)>0,\\
0, & \bar g(x)\le 0.
\end{cases}
\]
Therefore, $\dot h(0)\ge 0$ holds in all cases. Thus $h(\eta)=\eta$ is a barrier function for the set $\mathcal K=[0,\infty)$. Consequently, $\mathcal K$ is forward invariant under \eqref{eq:eta_dyn}, that is, $\eta(t)\ge 0, \forall t\ge 0$. 
\end{proof}

Fix $\tau>0$ and $\gamma>0$. We define the timed recovery control barrier function
\begin{equation}\label{eq:B_def}
B(x,\eta)\triangleq \gamma(\tau-\eta)-\bar g(x).
\end{equation}
Thus, the condition $B(x,\eta)\ge 0$ is equivalent to
\begin{equation}\label{eq:B_geom}
\bar g(x)\le \gamma(\tau-\eta),
\end{equation}
which shows that the admissible violation margin decreases as the timer approaches $\tau$. Then, the augmented safe set is defined as
\begin{equation}\label{eq:Saug_def}
\mathcal{S}_\tau\triangleq \{(x,\eta)\in\mathbb{R}^n\times\mathbb{R}_{\ge 0}\mid B(x,\eta)\ge 0\}.
\end{equation}

We define the admissible velocity set at $(x,\eta)$ by
\begin{equation}\label{eq:Ktau_def}
\begin{aligned}
\mathcal K_\tau(x,\eta)\triangleq
\bigl\{\xi\in\mathbb R^n \mid\
&\nabla g_i(x)^\top \xi
\le -\gamma\rho(x,\eta)+\alpha(B(x,\eta)),\\
&\forall i\in\mathcal I_{\max}(x)
\bigr\}.
\end{aligned}
\end{equation}
where $\alpha:\mathbb{R}\to\mathbb{R}$ is an extended class $\mathcal{K}$ function. The inequality condition is imposed only on the active maximizers of $\bar g$. 

\begin{lemma}[Directional derivative of aggregated violation]\label{lem:gbar_directional}
The function $\bar g$ is locally Lipschitz on $\mathbb{R}^n$. Moreover, for every $x\in\mathbb{R}^n$ and every direction $\xi\in\mathbb{R}^n$, its directional derivative is given by
\begin{equation}\label{eq:gbar_dir}
D\bar g(x;\xi)=\max_{i\in\mathcal{I}_{\max}(x)} \nabla g_i(x)^\top \xi.
\end{equation}
\end{lemma}

\begin{proof}
Since each $g_i$ is continuously differentiable and $\bar g$ is the pointwise maximum of finitely many continuously differentiable functions, $\bar g$ is locally Lipschitz. The directional derivative formula \eqref{eq:gbar_dir} is the standard directional derivative formula for the maximum of finitely many differentiable functions.
\end{proof}

\begin{remark}\label{rem:cbf_relation}
The recurrent safety requirement adopted here is strictly weaker than the standard CBF requirement of forward invariance. Indeed, any trajectory that remains in \(\mathcal C_g\) for all time is automatically \(\tau\)-recurrent for every \(\tau>0\).

Moreover, when \(x\in\mathcal C_g\) and \(\eta=0\), one has \(\rho(x,0)=0\), and the admissible velocity condition in \eqref{eq:Ktau_def} becomes
\[
\nabla g_i(x)^\top \xi \le \alpha(B(x,0)),
\quad \forall i\in\mathcal I_{\max}(x).
\]
Thus, on the feasible zero-timer section, the timed recovery control barrier condition reduces to a zeroing-CBF-type inequality, with an additional nonnegative slack determined by the recovery horizon \cite{ames2016control, ames2019control}. 
\end{remark}

\subsection{Theoretical Guarantees}
In this subsection, we establish the well-definedness and validity of our approach.
Since the timed recovery control barrier mechanism is only used on the augmented safe set \(\mathcal S_\tau\), it suffices to guarantee nonemptiness of \(\mathcal K_\tau(x,\eta)\) for states in \(\mathcal S_\tau\). This motivates the following assumption.

\begin{assumption}\label{ass:descent_direction_weak}
For every \(x\in\mathbb{R}^n\) satisfying \(\bar g(x)>0\), there exists a vector \(w(x)\in\mathbb{R}^n\) such that
\begin{equation}\label{eq:descent_direction}
\max_{i\in\mathcal{I}_{\max}(x)} \nabla g_i(x)^\top w(x)<0.
\end{equation}
\end{assumption}

\begin{theorem}[Well definedness of admissible velocity set]\label{thm:well_definedness}
Suppose Assumption \ref{ass:descent_direction_weak} holds. Then, for every \((x,\eta)\in\mathcal S_\tau\), the admissible velocity set \(\mathcal K_\tau(x,\eta)\) is nonempty. Consequently, the timed recovery control barrier condition is well defined on \(\mathcal S_\tau\).
\end{theorem}

\begin{proof}
Fix any \((x,\eta)\in\mathcal S_\tau\). We distinguish two cases.

If \(\bar g(x)>0\), then Assumption \ref{ass:descent_direction_weak} provides a vector \(w(x)\in\mathbb{R}^n\) such that
\[
c(x)\triangleq \max_{i\in\mathcal I_{\max}(x)} \nabla g_i(x)^\top w(x)<0.
\]
Let
\[
r(x,\eta)\triangleq -\gamma \rho(x,\eta)+\alpha(B(x,\eta)).
\]
Since \(r(x,\eta)\in\mathbb R\) and \(c(x)<0\), one has \(M c(x)\to -\infty\) as \(M\to+\infty\). Therefore, there exists \(M>0\) sufficiently large such that
\[
M c(x)\le r(x,\eta).
\]
Define \(\xi=Mw(x)\). Then, for every \(i\in\mathcal I_{\max}(x)\),
\[
\begin{aligned}
   \nabla g_i(x)^\top \xi
&=
M\,\nabla g_i(x)^\top w(x)
\le
M c(x)\\
&\le
r(x,\eta)
=
-\gamma \rho(x,\eta)+\alpha(B(x,\eta)). 
\end{aligned}
\]
Hence \(\xi\in\mathcal K_\tau(x,\eta)\), and therefore \(\mathcal K_\tau(x,\eta)\neq\varnothing\).

If \(\bar g(x)\le 0\), then by the definition of \(\rho\),
\[
\rho(x,\eta)=-\lambda\eta.
\]
Moreover, since \((x,\eta)\in\mathcal S_\tau\), one has \(B(x,\eta)\ge 0\), and because \(\alpha\) is an extended class \(\mathcal K\) function,
\[
\alpha(B(x,\eta))\ge 0.
\]
It follows that
\[
-\gamma \rho(x,\eta)+\alpha(B(x,\eta))
=
\gamma\lambda\eta+\alpha(B(x,\eta))
\ge 0.
\]
Therefore, with \(\xi=0\),
\[
\nabla g_i(x)^\top \xi
=
0
\le
-\gamma \rho(x,\eta)+\alpha(B(x,\eta)), \forall i\in\mathcal I_{\max}(x),
\]
which implies \(0\in\mathcal K_\tau(x,\eta)\). Hence \(\mathcal K_\tau(x,\eta)\neq\varnothing\).

Since the two cases exhaust all possibilities, the conclusion holds for every \((x,\eta)\in\mathcal S_\tau\).
\end{proof}

\begin{remark}\label{rem:descent_direction}
Assumption \ref{ass:descent_direction_weak} is a nondegeneracy condition on the aggregated violation function \(\bar g\) over the infeasible region. By Lemma \ref{lem:gbar_directional}, it requires that whenever \(\bar g(x)>0\), there exists a direction \(w\) such that
\[
D\bar g(x;w)
=
\max_{i\in\mathcal I_{\max}(x)} \nabla g_i(x)^\top w
<0.
\]
Thus, the assumption guarantees the existence of a strict first-order recovery direction precisely when the timed recovery mechanism is active. No such condition is needed on feasible states, since for \(\bar g(x)\le 0\), the proof of Theorem \ref{thm:well_definedness} already shows that \(\xi=0\) is admissible on \(\mathcal S_\tau\).

Furthermore, when multiple constraints are simultaneously active (i.e., \(|\mathcal I_{\max}(x)| > 1\)), this condition remains valid provided that the gradients of the active constraints do not positively span the origin. In the broader context of constrained nonlinear programming, this implies that Assumption \ref{ass:descent_direction_weak} is intrinsically satisfied whenever standard constraint qualifications, such as the Mangasarian-Fromovitz Constraint Qualification (MFCQ) \cite{nocedal2006numerical}, hold at the state \(x\). Consequently, this nondegeneracy condition is widely applicable and strictly avoids being overly restrictive in practical scenarios.
\end{remark}

\begin{theorem}[Forward invariance and $\tau$-recurrence]
\label{thm:main_recurrence}
For trajectory $(x(t),\eta(t))$ generated by \eqref{eq:general_dynamics} and \eqref{eq:eta_dyn}, if $\dot x(t)\in\mathcal K_\tau(x(t),\eta(t))\neq\varnothing$ and $(x(0),\eta(0))\in\mathcal S_\tau$, then 
\begin{equation}
    (x(t),\eta(t))\in \mathcal S_\tau, \quad \forall t\ge 0.
\end{equation}
Moreover, if in addition $x(0)\in\mathcal C_g$, then for every $t\ge 0$, there exists $s\in(t,t+\tau]$ such that
$x(s)\in\mathcal C_g$. 
\end{theorem}

\begin{proof}
We first show that the augmented safe set $\mathcal S_\tau$ is forward invariant. 
Using \eqref{eq:B_def}, the upper right Dini derivative of $B$ along the trajectory satisfies
\[
D^+B(x(t),\eta(t))
=
-\gamma \dot\eta(t)-D\bar g(x(t);\xi(t)).
\]
By Lemma \ref{lem:gbar_directional},
\[
D\bar g(x(t);\xi(t))
=
\max_{i\in\mathcal I_{\max}(x(t))}\nabla g_i(x(t))^\top \dot x(t).
\]
Since 
$\dot x(t)\in\mathcal K_\tau(x(t),\eta(t))$, one has
\[
\nabla g_i(x(t))^\top \dot x(t)
\le
-\gamma \rho(x(t),\eta(t))+\alpha(B(x(t),\eta(t)))
\]
for every $i\in\mathcal I_{\max}(x(t))$. Hence
\[
D\bar g(x(t);\xi(t))
\le
-\gamma \rho(x(t),\eta(t))+\alpha(B(x(t),\eta(t))).
\]
Using $\dot\eta(t)=\rho(x(t),\eta(t))$, we obtain
\[
\begin{aligned}
D^+B(x(t),\eta(t))
&=
-\gamma \dot\eta(t)-D\bar g(x(t);\xi(t))\\
&\ge
-\gamma \rho(x(t),\eta(t))
+\gamma \rho(x(t),\eta(t))\\
&-\alpha(B(x(t),\eta(t)))\\
&=
-\alpha(B(x(t),\eta(t))).
\end{aligned}
\]
Since $(x(0),\eta(0))\in\mathcal S_\tau$, one has $B(x(0),\eta(0))\ge 0$. Because $\alpha$ is an extended class $\mathcal K$ function, in particular $\alpha(0)=0$. Therefore, the scalar differential inequality above implies
\[
B(x(t),\eta(t))\ge 0,
\quad \forall t\ge 0.
\]
Hence
\[
(x(t),\eta(t))\in\mathcal S_\tau,
\quad \forall t\ge 0.
\]

We next prove the $\tau$-recurrence property. Since $\eta(0)=0$ by assumption and $\eta(\cdot)$ satisfies \eqref{eq:eta_dyn}, Lemma \ref{lem:eta_nonnegative} gives
\[
\eta(t)\ge 0,
\quad \forall t\ge 0.
\]

Assume now that $x(0)\in\mathcal C_g$. Suppose, by contradiction, that there exists $t_0\ge 0$ such that
\[
x(t)\notin \mathcal C_g,
\quad \forall t\in (t_0,t_0+\tau].
\]
Equivalently,
\[
\bar g(x(t))>0,
\quad \forall t\in (t_0,t_0+\tau].
\]
Then, by the timer dynamics \eqref{eq:eta_dyn},
\[
\dot\eta(t)=1,
\quad \forall t\in (t_0,t_0+\tau].
\]
Integrating over $(t_0,t_0+\tau]$ yields
\[
\eta(t_0+\tau)=\eta(t_0)+\tau\ge \tau,
\]
where the last inequality follows from $\eta(t_0)\ge 0$.

On the other hand, the forward invariance already established implies
\[
(x(t_0+\tau),\eta(t_0+\tau))\in\mathcal S_\tau,
\]
and hence
\[
0\le B(x(t_0+\tau),\eta(t_0+\tau))
=
\gamma\bigl(\tau-\eta(t_0+\tau)\bigr)-\bar g(x(t_0+\tau)).
\]
Since $\eta(t_0+\tau)\ge \tau$, it follows that
\[
\bar g(x(t_0+\tau))
\le
\gamma\bigl(\tau-\eta(t_0+\tau)\bigr)\le 0.
\]
Therefore, 
\[
x(t_0+\tau)\in\mathcal C_g,
\]
which contradicts the assumption that $x(t)\notin\mathcal C_g$ for all $t\in (t_0,t_0+\tau]$.

Hence, for every $t\ge 0$, there exists $s\in (t,t+\tau]$ such that
$x(s)\in\mathcal C_g$. This proves the $\tau$-recurrence of $\mathcal C_g$ along the closed-loop trajectory.
\end{proof}

Theorem \ref{thm:main_recurrence} shows that the proposed mechanism renders the augmented safe set forward invariant and guarantees \(\tau\)-recurrent recovery of the feasible set. This section therefore completes the constraint-level construction of the timed recovery mechanism and prepares for its incorporation into the optimization dynamics design.
\section{Constrained Nonlinear Programming via Timed Recovery Safe Gradient Flow}
\label{sec:optimization_dynamics}

In this section, we synthesize a timed recovery safe gradient flow for the inequality-constrained optimization problem \eqref{eq:nlp}. The basic idea is to interpret the unconstrained gradient descent as a nominal dynamics, and then introduce a feedback correction term so that the resulting closed-loop velocity satisfies the timed recovery control barrier condition developed in Section \ref{sec:rcbf}. In this way, the optimization objective and the recurrent safety requirement are incorporated into a unified control-theoretic framework.

We begin with the nominal gradient flow \cite{BROCKETT199179}
\begin{equation}\label{eq:nominal_gradient_flow}
\dot x = -\nabla f(x),
\end{equation}
which provides the steepest descent direction of the objective function but does not, in general, respect the recurrent safety requirement. To enforce the timed recovery control barrier condition, we introduce a correction input $u\in\mathbb R^n$, and consider the controlled system
\begin{equation}\label{eq:controlled_gradient_flow}
\dot x = -\nabla f(x)+u.
\end{equation}

Therefore, in order that the controlled dynamics satisfy the timed recovery control barrier condition, it is necessary and sufficient that
\[
-\nabla f(x)+u\in\mathcal K_\tau(x,\eta).
\]
This motivates the following admissible control set: 
\begin{equation}\label{eq:Utau_def}
\begin{aligned}
\mathcal U_\tau(x,\eta)\triangleq
\bigl\{u\in\mathbb R^n \mid\ 
&\nabla g_i(x)^\top u
\le \nabla g_i(x)^\top\nabla f(x)\\
&-\gamma\rho(x,\eta)+\alpha(B(x,\eta)),\\
&\forall i\in\mathcal I_{\max}(x)
\bigr\}.
\end{aligned}
\end{equation}
By construction,
\[
u\in\mathcal U_\tau(x,\eta)
\quad\Longleftrightarrow\quad
-\nabla f(x)+u\in\mathcal K_\tau(x,\eta).
\]

Among all admissible controls, we choose the minimum-norm one. This choice yields the least instantaneous deviation from the nominal gradient flow and therefore provides a canonical selection rule for the timed recovery safe dynamics. The timed recovery safe gradient flow is thus defined as follows.

\begin{definition}[Timed Recovery Safe Gradient Flow]\label{def:rsgf}
For each $(x,\eta)\in\mathcal S_\tau$, define the control law
\begin{equation}\label{eq:rsgf_qp_u}
\begin{aligned}
u_\tau(x,\eta)\in \arg\min_{u\in\mathbb R^n}\quad & \frac12\|u\|^2\\
\text{s.t.}\quad & u\in \mathcal U_\tau(x,\eta).
\end{aligned}
\end{equation}
The resulting timed recovery safe gradient flow is given by
\begin{subequations}\label{eq:rsgf_closed_loop}
\begin{align}
\dot x &= -\nabla f(x)+u_\tau(x,\eta),\\
\dot\eta &= \rho(x,\eta).
\end{align}
\end{subequations}
\end{definition}

The control synthesis problem \eqref{eq:rsgf_qp_u} is a convex quadratic program with linear inequality constraints. The next result shows that it is pointwise well-defined on the augmented safe set.

\begin{proposition}[Pointwise well definedness ]\label{prop:rsgf_wellposed}
Suppose Assumption \ref{ass:descent_direction_weak} holds. 
Then, for every $(x,\eta)\in\mathcal S_\tau$, the optimization problem \eqref{eq:rsgf_qp_u} admits a unique minimizer. 
\end{proposition}

\begin{proof}
Fix any $(x,\eta)\in\mathcal S_\tau$. By
Theorem \ref{thm:well_definedness},
the admissible velocity set $\mathcal K_\tau(x,\eta)$ is nonempty. 
It follows that $\mathcal U_\tau(x,\eta)$ is nonempty.

Moreover, $\mathcal U_\tau(x,\eta)$ is the intersection of finitely many closed half-spaces, and is therefore closed and convex. Since the objective function $u\mapsto \frac12\|u\|^2$ is continuous and strictly convex on $\mathbb R^n$, the optimization problem \eqref{eq:rsgf_qp_u} admits a unique minimizer.
\end{proof}

By construction, the minimum-norm control does not modify the nominal gradient flow unless such a modification is required by the timed recovery control barrier condition. Indeed, if $0\in\mathcal U_\tau(x,\eta)$, then the zero control is feasible for \eqref{eq:rsgf_qp_u}. Since the objective in \eqref{eq:rsgf_qp_u} is $\frac12\|u\|^2$, whose minimum possible value is attained at $u=0$, the unique optimizer satisfies
\[
u_\tau(x,\eta)=0.
\]
Hence, at any state $(x,\eta)\in\mathcal S_\tau$ for which the nominal velocity already satisfies the timed recovery control barrier condition, the timed recovery safe gradient flow coincides with the nominal gradient flow $\dot x=-\nabla f(x)$. This consistency property shows that the proposed control law acts only as a safety correction and remains inactive whenever the nominal descent direction is already admissible.

We now show that the control law \eqref{eq:rsgf_qp_u} indeed generates a closed-loop optimization flow that inherits the recurrent safety guarantee established in Section \ref{sec:rcbf}.

\begin{theorem}[Recurrent safety]\label{thm:rsgf_recurrent}
Let $(x(\cdot),\eta(\cdot))$ be a trajectory of \eqref{eq:rsgf_closed_loop}. If $(x(0),0)\in\mathcal S_\tau$, then $(x(t),\eta(t))\in\mathcal S_\tau$, $\forall t\ge 0$. Moreover, if $x(0)\in\mathcal C_g$, then for every $t\ge 0$, there exists $s\in (t,t+\tau]$ such that $x(s)\in\mathcal C_g$. 
\end{theorem}

\begin{proof}
By Definition \ref{def:rsgf}, one has
\[
u_\tau(x(t),\eta(t))\in\mathcal U_\tau(x(t),\eta(t))
\quad \text{for a.e. } t\ge 0.
\]
By \eqref{eq:Utau_def}, this is equivalent to
\[
-\nabla f(x(t))+u_\tau(x(t),\eta(t))\in \mathcal K_\tau(x(t),\eta(t)) \quad \text{for a.e. } t\ge 0.
\]
Since $(x(\cdot),\eta(\cdot))$ satisfies \eqref{eq:rsgf_closed_loop}, one has
\[
\dot x(t)=-\nabla f(x(t))+u_\tau(x(t),\eta(t)),
\]
and therefore
\[
\dot x(t)\in\mathcal K_\tau(x(t),\eta(t))
\quad \text{for a.e. } t\ge 0.
\]
The conclusion follows directly from Theorem \ref{thm:main_recurrence}.
\end{proof}

Theorem \ref{thm:rsgf_recurrent} shows that the optimization dynamics generated by \eqref{eq:rsgf_qp_u}-\eqref{eq:rsgf_closed_loop} are recurrently safe. We next record a simple structural property of equilibrium points, which will be used in the subsequent compatibility and stationarity analysis.

\begin{proposition}[Structure of equilibrium points]\label{prop:rsgf_equilibrium}
Let $(x^\ast,\eta^\ast)\in\mathcal S_\tau$ be an equilibrium point of \eqref{eq:rsgf_closed_loop}. Then
\begin{equation}\label{eq:eq_structure}
\eta^\ast=0,\quad \bar g(x^\ast)\le 0,
\end{equation}
and hence $x^\ast\in\mathcal C_g$. Moreover,
\begin{equation}\label{eq:eq_balance}
u_\tau(x^\ast,\eta^\ast)=\nabla f(x^\ast).
\end{equation}
\end{proposition}

\begin{proof}
Since $(x^\ast,\eta^\ast)$ is an equilibrium point of \eqref{eq:rsgf_closed_loop}, one has
\[
0=\dot x=-\nabla f(x^\ast)+u_\tau(x^\ast,\eta^\ast),
\]
which yields
\[
u_\tau(x^\ast,\eta^\ast)=\nabla f(x^\ast).
\]

It remains to analyze the timer equation. If $\bar g(x^\ast)>0$, then
\[
\dot\eta=\rho(x^\ast,\eta^\ast)=1,
\]
which contradicts the equilibrium condition $\dot\eta=0$. Hence $\bar g(x^\ast)\le 0$. In this case,
\[
0=\dot\eta=\rho(x^\ast,\eta^\ast)=-\lambda\eta^\ast.
\]
Since $\lambda>0$, it follows that $\eta^\ast=0$. Therefore $\bar g(x^\ast)\le 0$ and $\eta^\ast=0$, which proves \eqref{eq:eq_structure}. Because $x^\ast\in\mathcal C_g$ is equivalent to $\bar g(x^\ast)\le 0$, the proof is complete.
\end{proof}

Proposition \ref{prop:rsgf_equilibrium} shows that every equilibrium point of the timed recovery safe gradient flow lies on the feasible zero-timer section of the augmented safe set. Therefore, the equilibrium analysis of the closed-loop dynamics reduces to the pointwise optimality properties of the control synthesis problem \eqref{eq:rsgf_qp_u} evaluated at states satisfying $x\in\mathcal C_g$ and $\eta=0$. This observation provides the bridge from the feedback design developed here to the KKT analysis in the next section.

\section{System Behavior Analysis Under Fixed horizon}
\label{sec:system_behavior}
This section analyzes the qualitative behavior of the timed recovery safe gradient flow \eqref{eq:rsgf_closed_loop} for fixed horizon. The first subsection studies how the timer variable governs the alternation between descent and recovery, and the second examines compatibility with the nominal gradient flow together with the role of \(\tau\). Throughout, let \((x(\cdot),\eta(\cdot))\) be a trajectory of \eqref{eq:rsgf_closed_loop} with \((x(0),0)\in\mathcal S_\tau\). 

\subsection{Qualitative System Behavior Under Fixed horizon}

\begin{proposition}
\label{prop:timer_phase_dynamics}
Suppose Assumption \ref{ass:descent_direction_weak} holds, and let $(x(\cdot),\eta(\cdot))$ be an absolutely continuous trajectory of \eqref{eq:rsgf_closed_loop} with $(x(0),0)\in\mathcal S_\tau$. Then the following statements hold.

\begin{enumerate}
\item For all $t\ge 0$,
\[
0\le \eta(t)\le \tau,
\quad
\bar g(x(t))\le \gamma\bigl(\tau-\eta(t)\bigr)\le \gamma\tau.
\]
In particular, $\bar g(x(t))>0$ implies $\eta(t)<\tau$.

\item Let $[t_a,t_b]\subset[0,\infty)$. If $\bar g(x(t))\le 0$ for all $t\in[t_a,t_b]$, then
\[
\bar g(x(t))
\le
\gamma\Bigl(\tau-\eta(t_a)e^{-\lambda(t-t_a)}\Bigr),
\quad t\in[t_a,t_b].
\]

\item Let $[t_a,t_b]\subset[0,\infty)$. If $\bar g(x(t))>0$ for all $t\in[t_a,t_b]$, then
\[
\bar g(x(t))
\le
\gamma\Bigl(\tau-\eta(t_a)-(t-t_a)\Bigr),
\quad t\in[t_a,t_b],
\]
and consequently $t_b-t_a<\tau-\eta(t_a)$.
\end{enumerate}
\end{proposition}

\begin{proof}
By Theorem \ref{thm:rsgf_recurrent}, one has $(x(t),\eta(t))\in\mathcal S_\tau$ for all $t\ge 0$. Since $\eta(0)=0$, Lemma \ref{lem:eta_nonnegative} yields $\eta(t)\ge 0$ for all $t\ge 0$.

To prove the upper bound $\eta(t)\le \tau$, consider the boundary $\eta=\tau$. If $(x,\tau)\in\mathcal S_\tau$, then
\[
0\le B(x,\tau)=\gamma(\tau-\tau)-\bar g(x)=-\bar g(x),
\]
so $\bar g(x)\le 0$. Hence, by \eqref{eq:eta_dyn}, one has $\dot\eta=-\lambda\tau<0$ at every such boundary point. Since $\eta(0)=0<\tau$, it follows that $\eta(t)\le \tau$ for all $t\ge 0$. Therefore $0\le \eta(t)\le \tau$ for all $t\ge 0$.

Again from $(x(t),\eta(t))\in\mathcal S_\tau$, one has
\[
0\le B(x(t),\eta(t))
=
\gamma\bigl(\tau-\eta(t)\bigr)-\bar g(x(t)),
\]
which gives
\[
\bar g(x(t))\le \gamma\bigl(\tau-\eta(t)\bigr)\le \gamma\tau,
\quad \forall t\ge 0.
\]
In particular, if $\bar g(x(t))>0$, then necessarily $\eta(t)<\tau$.

Now let $[t_a,t_b]\subset[0,\infty)$. If $\bar g(x(t))\le 0$ for all $t\in[t_a,t_b]$, then \eqref{eq:eta_dyn} reduces to $\dot\eta(t)=-\lambda\eta(t)$, so
\[
\eta(t)=\eta(t_a)e^{-\lambda(t-t_a)},
\quad \forall t\in[t_a,t_b].
\]
Substituting this into $\bar g(x(t))\le \gamma(\tau-\eta(t))$ proves the feasible-segment estimate.

If $\bar g(x(t))>0$ for all $t\in[t_a,t_b]$, then \eqref{eq:eta_dyn} gives $\dot\eta(t)=1$, hence
\[
\eta(t)=\eta(t_a)+(t-t_a),
\quad \forall t\in[t_a,t_b].
\]
Substituting this into $\bar g(x(t))\le \gamma(\tau-\eta(t))$ yields
\[
\bar g(x(t))
\le
\gamma\Bigl(\tau-\eta(t_a)-(t-t_a)\Bigr),
\quad \forall t\in[t_a,t_b].
\]
Since $\bar g(x(t_b))>0$, one has $\eta(t_b)<\tau$, and therefore
\[
\eta(t_a)+(t_b-t_a)=\eta(t_b)<\tau.
\]
Thus $t_b-t_a<\tau-\eta(t_a)$, as claimed.
\end{proof}

Proposition \ref{prop:timer_phase_dynamics} describes the segmentwise evolution of the timer state and the associated recovery margin. In particular, along infeasible segments, the remaining margin \(\tau-\eta\) decreases monotonically, so the timed recovery constraint becomes progressively more restrictive. The following lemma makes this effect precise by showing how the tightening of the constraint yields a restorative bound on the instantaneous evolution of \(\bar g\).

\begin{lemma}[Restorative mechanism of timed recovery CBF]
\label{prop:restorative_mechanism}
For the closed-loop system \eqref{eq:rsgf_closed_loop}, the following statements hold for almost every $t\ge 0$.

\begin{enumerate}
\item One has
\[
D^+\bar g(x(t))
\le
-\gamma \rho(x(t),\eta(t))+\alpha(B(x(t),\eta(t))).
\]\label{restorative_mechanism_item1}

\item If $\bar g(x(t))>0$, then
\[
D^+\bar g(x(t))
\le
-\gamma+\alpha(B(x(t),\eta(t))).
\]
Moreover, if $\alpha(B(x(t),\eta(t)))<\gamma$, then $D^+\bar g(x(t))<0$. In particular, if $B(x(t),\eta(t))=0$ and $\bar g(x(t))>0$, then
\[
D^+\bar g(x(t))\le -\gamma<0.
\]\label{restorative_mechanism_item2}

\item If $x(t)\in\partial\mathcal C_g$ and $\eta(t)=0$, then $D^+\bar g(x(t))\le \alpha(\gamma\tau)$. \label{restorative_mechanism_item3}
\end{enumerate}
\end{lemma}

\begin{proof}
For the closed-loop system \eqref{eq:rsgf_closed_loop}, one has for almost every $t\ge 0$,
\[
\dot x(t)=-\nabla f(x(t))+u_\tau(x(t),\eta(t))
\in \mathcal K_\tau(x(t),\eta(t)).
\]
By Lemma \ref{lem:gbar_directional},
\[
D^+\bar g(x(t))
=
D\bar g(x(t);\dot x(t))
=
\max_{i\in\mathcal I_{\max}(x(t))}\nabla g_i(x(t))^\top \dot x(t).
\]
Since $\dot x(t)\in\mathcal K_\tau(x(t),\eta(t))$, the defining inequality of $\mathcal K_\tau$ yields
\[
\nabla g_i(x(t))^\top \dot x(t)
\le
-\gamma\rho(x(t),\eta(t))+\alpha(B(x(t),\eta(t)))
\]
for every $i\in\mathcal I_{\max}(x(t))$. Taking the maximum over $i\in\mathcal I_{\max}(x(t))$ proves item \ref{restorative_mechanism_item1}.

If $\bar g(x(t))>0$, then by the definition of $\rho$, one has $\rho(x(t),\eta(t))=1$. Substituting this into item \ref{restorative_mechanism_item1} gives
\[
D^+\bar g(x(t))
\le
-\gamma+\alpha(B(x(t),\eta(t))),
\]
which proves the first claim in item \ref{restorative_mechanism_item2}. If in addition $\alpha(B(x(t),\eta(t)))<\gamma$, then
\[
-\gamma+\alpha(B(x(t),\eta(t)))<0,
\]
and hence $D^+\bar g(x(t))<0$. Finally, if $B(x(t),\eta(t))=0$ and $\bar g(x(t))>0$, then $\alpha(B(x(t),\eta(t)))=\alpha(0)=0$, so
\[
D^+\bar g(x(t))\le -\gamma<0.
\]

For item \ref{restorative_mechanism_item3}, if $x(t)\in\partial\mathcal C_g$ and $\eta(t)=0$, then $\bar g(x(t))=0$, and therefore $\rho(x(t),0)=0$. Moreover,
\[
B(x(t),0)=\gamma\tau-\bar g(x(t))=\gamma\tau.
\]
Substituting these identities into the estimate in item \ref{restorative_mechanism_item1} yields
\[
D^+\bar g(x(t))\le \alpha(\gamma\tau).
\]
This proves item \ref{restorative_mechanism_item3}.
\end{proof}

Lemma \ref{prop:restorative_mechanism} shows that the fixed horizon design permits temporary exits from \(\mathcal C_g\), since \(D^+\bar g(x)\) is not required to be nonpositive on \(\partial\mathcal C_g\). However, this loss of feasibility is quantitatively regulated: once \(\bar g(x)>0\), the tightening of the timed recovery constraint eventually forces \(D^+\bar g(x)<0\) when the remaining recurrence budget becomes small enough. Hence the mechanism does not prevent boundary crossing, but it does enforce restorative motion before the recovery budget is exhausted. 

\begin{remark}
\label{rem:objective_variation_revised}
For the closed-loop system \eqref{eq:rsgf_closed_loop}, one has for almost every \(t\ge 0\),
\begin{equation}\label{eq:f_dot_revised}
\frac{d}{dt}f(x(t))
=
-\|\nabla f(x(t))\|^2+\nabla f(x(t))^\top u_\tau(x(t),\eta(t)).
\end{equation}
Hence, under the fixed horizon timed recovery mechanism, the objective function \(f\) is generally not monotone along trajectories, since the sign of \(\nabla f(x)^\top u_\tau(x,\eta)\) is not determined a priori. Therefore, unlike in the nominal gradient flow, \(f\) alone cannot be used to certify monotone descent or serve as a Lyapunov function for the augmented closed-loop dynamics.
\end{remark}

Now proceed to characterize the invariant structures of the system. 
\begin{proposition}[Necessary structure of periodic orbits]
\label{prop:periodic_orbit_structure}
Let $\Gamma\subset \mathcal S_\tau$ be a periodic orbit of \eqref{eq:rsgf_closed_loop}. Assume that
\[
\Gamma\not\subset \mathcal C_g\times\{0\}.
\]
Then $\Gamma$ intersects both
\[
A\triangleq \{(x,\eta)\in\mathcal S_\tau \mid \bar g(x)>0\}
\]
and
\[
B\triangleq \{(x,\eta)\in\mathcal S_\tau \mid \bar g(x)\le 0,\ \eta>0\}.
\]
Consequently, $\Gamma$ intersects the switching surface
\[
\{(x,\eta)\in\mathcal S_\tau \mid \bar g(x)=0\},
\]
and each full period contains an infeasible phase together with a feasible recovery phase.
\end{proposition}

\begin{proof}
Since $\Gamma\subset \mathcal S_\tau$ and $\Gamma\not\subset \mathcal C_g\times\{0\}$, there exists $(x,\eta)\in\Gamma$ such that either $\bar g(x)>0$ or $\eta>0$. Therefore,
\[
\Gamma\cap (A\cup B)\neq \varnothing.
\]

We first show that $\Gamma\cap A\neq \varnothing$ implies $\Gamma\cap B\neq \varnothing$. Let $(x_\Gamma(\cdot),\eta_\Gamma(\cdot))$ be a periodic parametrization of $\Gamma$, and choose $t_a$ such that
\[
(x_\Gamma(t_a),\eta_\Gamma(t_a))\in A.
\]
If $\Gamma\subset A$, then $\bar g(x_\Gamma(t))>0$ for all $t$, and hence $\dot\eta_\Gamma(t)=1$ along the entire orbit by \eqref{eq:eta_dyn}. This is impossible for a periodic orbit, since $\eta_\Gamma$ would then be strictly increasing. Therefore $\Gamma\not\subset A$, so there exists a first time $t_b>t_a$ such that
\[
\bar g(x_\Gamma(t_b))\le 0, 
\quad
\bar g(x_\Gamma(t))>0 \ \text{for all } t\in [t_a,t_b).
\]
On $[t_a,t_b)$, one has $\dot\eta_\Gamma(t)=1$, and therefore
\[
\eta_\Gamma(t_b)=\eta_\Gamma(t_a)+(t_b-t_a)>0.
\]
Since also $\bar g(x_\Gamma(t_b))\le 0$, it follows that $(x_\Gamma(t_b),\eta_\Gamma(t_b))\in B$. Hence $\Gamma\cap B\neq \varnothing$.

We next show that $\Gamma\cap B\neq \varnothing$ implies $\Gamma\cap A\neq \varnothing$. Suppose, by contradiction, that $\Gamma\cap B\neq \varnothing$ and $\Gamma\cap A=\varnothing$. Then $\bar g(x)\le 0$ for every $(x,\eta)\in\Gamma$, so the whole orbit is feasible. Choose $t_c$ such that
\[
(x_\Gamma(t_c),\eta_\Gamma(t_c))\in B.
\]
Then $\eta_\Gamma(t_c)>0$, and since the orbit is feasible everywhere, the timer dynamics reduce to $\dot\eta_\Gamma(t)=-\lambda \eta_\Gamma(t)$ for all $t$. 
Hence
\[
\eta_\Gamma(t)=\eta_\Gamma(t_c)e^{-\lambda (t-t_c)}
\quad \forall t\ge t_c.
\]
In particular, $\eta_\Gamma(t)>0$ for all $t\ge t_c$, and $\eta_\Gamma$ is strictly decreasing. This contradicts periodicity. Therefore $\Gamma\cap A\neq \varnothing$.

We have thus proved that $\Gamma$ intersects both $A$ and $B$. Since the map
\[
t\mapsto \bar g(x_\Gamma(t))
\]
is continuous, any transition between the regions $\{\bar g>0\}$ and $\{\bar g\le 0\}$ must pass through $\{\bar g=0\}$. Hence $\Gamma$ intersects the switching surface
\[
\{(x,\eta)\in\mathcal S_\tau \mid \bar g(x)=0\}.
\]

Finally, because $\Gamma\cap A\neq \varnothing$, each period contains a nontrivial interval on which $\bar g(x_\Gamma(t))>0$, namely an infeasible phase. Likewise, because $\Gamma\cap B\neq \varnothing$, each period also contains a nontrivial interval on which $\bar g(x_\Gamma(t))\le 0$ and $\eta_\Gamma(t)>0$, namely a feasible recovery phase. This proves the final statement.
\end{proof}

Having characterized the necessary structure of nontrivial periodic behavior, we now examine the opposite regime of stationary solutions. The next result shows that, on the feasible interior where the correction remains inactive, interior equilibria of the fixed horizon closed-loop system are exactly the stationary, equivalently KKT, points of the original problem.

\begin{proposition}
\label{prop:fixed_tau_interior_kkt}
Assume that $0\in \mathcal U_\tau(x,0)$ for all $x\in \operatorname{int}(\mathcal C_g)$. Then, for every $x^\star\in \operatorname{int}(\mathcal C_g)$, the following statements are equivalent:
\begin{enumerate}
\item $(x^\star,0)$ is an equilibrium of \eqref{eq:rsgf_closed_loop}.
\item $\nabla f(x^\star)=0$.
\item $x^\star\in X_{\mathrm{KKT}}$.
\end{enumerate}
\end{proposition}

\begin{proof}
We prove \((i)\Rightarrow(ii)\Rightarrow(iii)\Rightarrow(i)\).

\((i)\Rightarrow(ii)\): If \((x^\star,0)\) is an equilibrium of \eqref{eq:rsgf_closed_loop}, then
\[
0=-\nabla f(x^\star)+u_\tau(x^\star,0).
\]
Since \(x^\star\in \operatorname{int}(\mathcal C_g)\), the assumption gives \(0\in \mathcal U_\tau(x^\star,0)\). Because \(u_\tau(x^\star,0)\) is the minimum-norm element of \(\mathcal U_\tau(x^\star,0)\), one has \(u_\tau(x^\star,0)=0\). Hence \(\nabla f(x^\star)=0\).

\((ii)\Rightarrow(iii)\): Since \(x^\star\in \operatorname{int}(\mathcal C_g)\subset \mathcal C_g\), one has \(g(x^\star)\le 0\). Choosing \(p^\star=0\in\mathbb R^m\), the KKT conditions reduce to \(\nabla f(x^\star)=0\), which holds by assumption. Therefore \(x^\star\in X_{\mathrm{KKT}}\).

\((iii)\Rightarrow(i)\): Let \(x^\star\in X_{\mathrm{KKT}}\cap \operatorname{int}(\mathcal C_g)\). Then there exists \(p^\star\in\mathbb R^m_{\ge 0}\) such that
\[
\nabla f(x^\star)+\frac{\partial g(x^\star)}{\partial x}^{\top}p^\star=0,
\quad
(p^\star)^\top g(x^\star)=0.
\]
If \(g_i(x^\star)<0\), then complementarity gives \(p_i^\star=0\). If \(g_i(x^\star)=0\), then, because \(x^\star\in \operatorname{int}(\mathcal C_g)\), there exists \(r>0\) such that \(B_r(x^\star)\subset \mathcal C_g\), and hence
\[
g_i(x)\le 0=g_i(x^\star),
\quad \forall x\in B_r(x^\star).
\]
Thus \(x^\star\) is a local maximizer of \(g_i\), so \(\nabla g_i(x^\star)=0\). Therefore, in either case,
\[
p_i^\star \nabla g_i(x^\star)=0,
\quad i=1,\dots,m.
\]
Hence
\[
\frac{\partial g(x^\star)}{\partial x}^{\top}p^\star
=
\sum_{i=1}^m p_i^\star \nabla g_i(x^\star)
=
0,
\]
and the KKT stationarity condition reduces to
\[
\nabla f(x^\star)=0.
\]

Since \(x^\star\in \operatorname{int}(\mathcal C_g)\), the assumption gives \(0\in \mathcal U_\tau(x^\star,0)\), and therefore \(u_\tau(x^\star,0)=0\). Moreover, \(x^\star\in\mathcal C_g\) implies \(\bar g(x^\star)\le 0\), so
\[
\rho(x^\star,0)=0.
\]
Consequently,
\[
-\nabla f(x^\star)+u_\tau(x^\star,0)=0,
\quad
\rho(x^\star,0)=0,
\]
and hence \((x^\star,0)\) is an equilibrium of \eqref{eq:rsgf_closed_loop}.
\end{proof}

Proposition \ref{prop:fixed_tau_interior_kkt} identifies the interior equilibria of the fixed horizon closed-loop system. We next strengthen this static characterization to a local dynamical one, and show that around a strict interior KKT point the recurrent correction remains inactive, so the closed-loop dynamics locally reduces to the nominal gradient flow and inherits its exponential convergence behavior.

\begin{theorem}[Local convergence to a interior KKT point]
\label{thm:fixed_tau_local_interior_convergence}
Assume $f\in C^2$ and $0\in \mathcal U_\tau(x,0)$ for all $x\in \operatorname{int}(\mathcal C_g)$. Let $x^\star\in \operatorname{int}(\mathcal C_g)$ satisfy $\nabla f(x^\star)=0$ and $\nabla^2 f(x^\star)\succ 0$. Then there exist neighborhoods $x^\star\in U_1\Subset U_0\Subset \operatorname{int}(\mathcal C_g)$ 
and constants $c>0$, $\mu>0$ such that every trajectory of \eqref{eq:rsgf_closed_loop} with $x(0)\in U_1$ and $\eta(0)=0$ satisfies
\[
u_\tau(x(t),\eta(t))=0,\quad \eta(t)=0,\quad x(t)\in U_0,\quad \forall t\ge 0,
\]
and
\[
\|x(t)-x^\star\|\le c e^{-\mu t}\|x(0)-x^\star\|,\quad \forall t\ge 0.
\]
Consequently, $(x(t),\eta(t))\to (x^\star,0)$ and $x^\star\in X_{\mathrm{KKT}}$.
\end{theorem}

\begin{proof}
Since $\nabla^2 f(x^\star)\succ 0$, the equilibrium $x^\star$ of the nominal gradient flow $\dot x=-\nabla f(x)$ is locally exponentially stable. Hence there exist neighborhoods $x^\star\in U_1\Subset U_0\Subset \operatorname{int}(\mathcal C_g)$ and constants $c>0$, $\mu>0$ such that every solution of $\dot x=-\nabla f(x)$ with $x(0)\in U_1$ remains in $U_0$ for all $t\ge 0$ and satisfies
\[
\|x(t)-x^\star\|\le c e^{-\mu t}\|x(0)-x^\star\|,\quad \forall t\ge 0.
\]

Now consider the fixed horizon timed recovery safe gradient flow \eqref{eq:rsgf_closed_loop} with $x(0)\in U_1$ and $\eta(0)=0$. Because $U_0\subset \operatorname{int}(\mathcal C_g)$ and $0\in \mathcal U_\tau(x,0)$ for every $x\in \operatorname{int}(\mathcal C_g)$, the minimum-norm property gives $u_\tau(x,0)=0$ for all $x\in U_0$. Hence, as long as $x(t)\in U_0$ and $\eta(t)=0$, the closed-loop system reduces to $\dot x(t)=-\nabla f(x(t))$. Since $U_0\subset \operatorname{int}(\mathcal C_g)$, one has $\bar g(x(t))<0$ on $U_0$, and therefore $\dot\eta(t)=\rho(x(t),0)=0$. Thus $\eta(t)=0$ for all $t\ge 0$, and the closed-loop trajectory coincides exactly with the nominal gradient flow. The exponential estimate follows from the local exponential stability of $x^\star$ under $\dot x=-\nabla f(x)$.

Finally, Proposition \ref{prop:fixed_tau_interior_kkt} yields $x^\star\in X_{\mathrm{KKT}}$.
\end{proof}

Overall, the fixed horizon design yields a recurrently regulated optimization dynamics characterized by two coupled phases: nominal descent when the correction is inactive, and timer-driven recovery once infeasibility occurs. The timer dynamics progressively tighten the admissible violation margin during infeasible motion and thereby enforce return to the feasible set within the prescribed horizon. Accordingly, the defining behavior of the fixed horizon closed loop is bounded-duration feasibility loss together with guaranteed recurrent recovery, rather than monotone descent. We next study how this qualitative picture depends on the choice of the recovery horizon \(\tau\). 

\subsection{Compatibility and Horizon-Induced Tradeoff}
The preceding subsection describes the generic behavior induced by the fixed horizon recurrent mechanism. We now focus on regimes in which this mechanism becomes inactive or weakens as the recovery horizon varies. These results clarify how the proposed dynamics remains compatible with the nominal gradient flow and how the horizon parameter controls the balance between feasibility recovery and optimization freedom. Before stating the zero-horizon compatibility result, we introduce the CBF-type correction associated with the aggregated constraint \(\bar g(x)\le 0\). For \(x\in\mathcal C_g\), define
\[
\begin{aligned}
\mathcal U_0(x)\triangleq
\bigl\{u\in\mathbb R^n \mid\ 
&\nabla g_i(x)^\top u
\le \nabla g_i(x)^\top\nabla f(x)\\
&+\alpha(-\bar g(x)),\ 
\forall i\in\mathcal I_{\max}(x)
\bigr\}.
\end{aligned}
\]
When \(\mathcal U_0(x)\neq\varnothing\), define
\[
u_0(x)\triangleq
\arg\min_{u\in\mathcal U_0(x)}\frac12\|u\|^2 .
\]
The corresponding velocity field
\[
\dot x=-\nabla f(x)+u_0(x)
\]
is referred to here as the aggregated-constraint CBF-type safe gradient flow, since it is obtained by the minimum-norm correction of the nominal gradient flow subject to the zeroing-CBF-type condition induced by \(\bar g(x)\le 0\).

\begin{proposition}[Compatibility with safe gradient flow]
\label{prop:tau_zero_safe}
Fix \(x\in\mathcal C_g\) such that \(\mathcal U_0(x)\neq\varnothing\), and set \(\eta=0\). For each \(\tau>0\), let \(u_\tau(x,0)\) be defined by \eqref{eq:rsgf_qp_u}. Then
\[
\lim_{\tau\to 0}u_\tau(x,0)=u_0(x).
\]
Consequently, the zero-horizon limit of the closed-loop velocity is the aggregated-constraint CBF-type safe gradient flow
\[
\dot x=-\nabla f(x)+u_0(x).
\]
\end{proposition}

\begin{proof}
Fix \(x\in\mathcal C_g\) and set \(\eta=0\). Since \(\bar g(x)\le0\), one has \(\rho(x,0)=0\) and
\[
B_\tau(x,0)=\gamma\tau-\bar g(x).
\]
Hence the admissible control set \(\mathcal U_\tau(x,0)\) is given by
\[
\nabla g_i(x)^\top u
\le
\nabla g_i(x)^\top \nabla f(x)
+\alpha(\gamma\tau-\bar g(x)),
\forall i\in\mathcal I_{\max}(x).
\]
Since \(\alpha\) is continuous and strictly increasing with \(\alpha(0)=0\),
\[
\alpha(\gamma\tau-\bar g(x))
\to
\alpha(-\bar g(x)),
\quad \tau\to 0.
\]
Therefore, the defining half-spaces of \(\mathcal U_\tau(x,0)\) converge to those of \(\mathcal U_0(x)\).

Moreover, because \(\gamma\tau-\bar g(x)\ge -\bar g(x)\), monotonicity of \(\alpha\) gives
\[
\mathcal U_0(x)\subseteq \mathcal U_\tau(x,0),
\quad \forall \tau>0.
\]
Thus \(u_0(x)\in\mathcal U_\tau(x,0)\), and by the minimum-norm property,
\[
\|u_\tau(x,0)\|\le \|u_0(x)\|.
\]
Hence every sequence \(\tau_k\to 0\) admits a convergent subsequence \(u_{\tau_k}(x,0)\to \bar u\). Passing to the limit in the linear inequalities defining \(\mathcal U_{\tau_k}(x,0)\) yields \(\bar u\in\mathcal U_0(x)\). Also,
\[
\|\bar u\|\le \|u_0(x)\|.
\]
Since \(u_0(x)\) is the unique minimum-norm element of the closed convex set \(\mathcal U_0(x)\), it follows that \(\bar u=u_0(x)\). Since every convergent subsequence has the same limit, the full limit holds:
\[
\lim_{\tau\to 0}u_\tau(x,0)=u_0(x).
\]

Finally, by the definition of \(\mathcal U_0(x)\),
\[
\nabla g_i(x)^\top\bigl(-\nabla f(x)+u_0(x)\bigr)
\le
\alpha(-\bar g(x)),
\quad
\forall i\in\mathcal I_{\max}(x).
\]
Using \eqref{eq:gbar_dir}, we obtain
\begin{equation*}
\begin{aligned}
&D\bar g\bigl(x;-\nabla f(x)+u_0(x)\bigr)\\
&\quad =
\max_{i\in\mathcal I_{\max}(x)}
\nabla g_i(x)^\top
\bigl(-\nabla f(x)+u_0(x)\bigr)\\
&\quad \le \alpha(-\bar g(x)).
\end{aligned}
\end{equation*}
If \(x\in\partial\mathcal C_g\), then \(\bar g(x)=0\), and hence
\[
D\bar g\bigl(x;-\nabla f(x)+u_0(x)\bigr)\le \alpha(0)=0.
\]
This is precisely the tangent-cone condition ensuring instantaneous nonescape from \(\mathcal C_g\) at the boundary.
\end{proof}

Proposition \ref{prop:tau_zero_safe} identifies the zero-horizon limit of the fixed horizon design, in which the recurrent mechanism reduces to a CBF-type safe gradient flow. The next result considers the opposite limit \(\tau\to+\infty\), and shows that the corrective term vanishes asymptotically, so that the nominal gradient flow is recovered.

\begin{proposition}[Compatibility with nominal gradient flow]
\label{prop:tau_infty_nominal}
Assume that \(\alpha:\mathbb{R}\to\mathbb{R}\) satisfies $\lim_{s\to+\infty}\alpha(s)=+\infty$. 
Then, for every fixed augmented state \((x,\eta)\in\mathbb{R}^n\times\mathbb{R}_{\ge 0}\),
\[
\lim_{\tau\to+\infty}u_\tau(x,\eta)=0.
\]
Consequently,
\[
\lim_{\tau\to+\infty}\bigl(-\nabla f(x)+u_\tau(x,\eta)\bigr)
=
-\nabla f(x).
\]
\end{proposition}

\begin{proof}
Fix any \((x,\eta)\in\mathbb{R}^n\times\mathbb{R}_{\ge 0}\). Since \(\mathcal I_{\max}(x)\) is finite, the quantity
\[
c(x,\eta)
\triangleq
\max_{i\in\mathcal I_{\max}(x)}
\Bigl(
\nabla g_i(x)^\top(-\nabla f(x))
+\gamma \rho(x,\eta)
\Bigr)
\]
is a finite real number.

By the assumption that \(\alpha\) is unbounded above on \([0,\infty)\), there exists a scalar \(s_{x,\eta}\ge 0\) such that
\[
\alpha(s)\ge c(x,\eta),
\quad \forall s\ge s_{x,\eta}.
\]
Now define
\[
\tau_{x,\eta}
\triangleq
\max\left\{0,\frac{s_{x,\eta}+\gamma\eta+\bar g(x)}{\gamma}\right\}.
\]
Then, for every \(\tau\ge \tau_{x,\eta}\),
\[
B(x,\eta)
=
\gamma(\tau-\eta)-\bar g(x)
\ge s_{x,\eta},
\]
and therefore
\[
\alpha(B(x,\eta))\ge c(x,\eta).
\]

Hence, for every \(i\in\mathcal I_{\max}(x)\),
\[
\nabla g_i(x)^\top(-\nabla f(x))
+\gamma \rho(x,\eta)
\le c(x,\eta)
\le \alpha(B(x,\eta)),
\]
which is equivalent to
\[
\nabla g_i(x)^\top(-\nabla f(x))
\le
-\gamma \rho(x,\eta)+\alpha(B(x,\eta)).
\]
By the definition of \(\mathcal K_\tau(x,\eta)\), this shows that
\[
-\nabla f(x)\in\mathcal K_\tau(x,\eta).
\]
Equivalently, by the definition of \(\mathcal U_\tau(x,\eta)\),
\[
0\in\mathcal U_\tau(x,\eta).
\]

Since \(u_\tau(x,\eta)\) is defined as the minimum-norm element of \(\mathcal U_\tau(x,\eta)\), and \(0\) is feasible, it follows that
\[
u_\tau(x,\eta)=0,
\quad \forall \tau\ge \tau_{x,\eta}.
\]
Therefore,
\[
\lim_{\tau\to+\infty}u_\tau(x,\eta)=0.
\]
The second conclusion follows immediately:
\[
\lim_{\tau\to+\infty}\bigl(-\nabla f(x)+u_\tau(x,\eta)\bigr)
=
-\nabla f(x).
\]
\end{proof}

Proposition \ref{prop:tau_infty_nominal} characterizes the asymptotic effect of enlarging the recovery horizon to its extreme limit: as \(\tau\to+\infty\), the recurrent correction vanishes and the closed-loop dynamics degenerates pointwise to the nominal gradient flow. We next move from this asymptotic statement to a finite-\(\tau\) comparison. The following result shows how increasing the recovery horizon enlarges the admissible sets, weakens the minimum-norm correction, and correspondingly relaxes the recovery requirement at the feasible boundary. 

\begin{theorem}[Recovery-exploration tradeoff]
\label{prop:tau_tradeoff}
Let \(0<\tau_1\le \tau_2\), and fix any augmented state $(x,\eta)\in\mathbb{R}^n\times\mathbb{R}_{\ge 0}$. For each \(\tau>0\), define
\[
B_\tau(x,\eta)\triangleq \gamma(\tau-\eta)-\bar g(x),
\]
and let \(\mathcal S_\tau\), \(\mathcal K_\tau(x,\eta)\), \(\mathcal U_\tau(x,\eta)\), and \(u_\tau(x,\eta)\) be defined as in
\eqref{eq:Saug_def}, \eqref{eq:Ktau_def}, \eqref{eq:Utau_def}, and \eqref{eq:rsgf_qp_u}, respectively.
Then the following statements hold.
\begin{enumerate}
\item
\[
\mathcal S_{\tau_1}\subseteq \mathcal S_{\tau_2}, 
\mathcal K_{\tau_1}(x,\eta)\subseteq \mathcal K_{\tau_2}(x,\eta), 
\mathcal U_{\tau_1}(x,\eta)\subseteq \mathcal U_{\tau_2}(x,\eta).
\]
\item If the quadratic programs defining \(u_{\tau_1}(x,\eta)\) and \(u_{\tau_2}(x,\eta)\) are well posed, then $\|u_{\tau_2}(x,\eta)\|\le \|u_{\tau_1}(x,\eta)\|$. 
\item If \(x\in\partial\mathcal C_g\) and \(\eta=0\), and if
\[
\xi_{\tau_j}\triangleq -\nabla f(x)+u_{\tau_j}(x,0),
\quad j=1,2,
\]
then
\[
D\bar g(x;\xi_{\tau_j})\le \alpha(\gamma\tau_j),
\quad j=1,2.
\]
\end{enumerate} 
\end{theorem}

\begin{proof}
Since \(0<\tau_1\le \tau_2\), one has
\[
B_{\tau_2}(x,\eta)-B_{\tau_1}(x,\eta)
=
\gamma(\tau_2-\tau_1)\ge 0.
\]
Hence
\[
B_{\tau_2}(x,\eta)\ge B_{\tau_1}(x,\eta).
\]

We first prove \(\mathcal S_{\tau_1}\subseteq \mathcal S_{\tau_2}\).
If \((x,\eta)\in\mathcal S_{\tau_1}\), then by definition
\[
B_{\tau_1}(x,\eta)\ge 0.
\]
Since \(B_{\tau_2}(x,\eta)\ge B_{\tau_1}(x,\eta)\), it follows that
\[
B_{\tau_2}(x,\eta)\ge 0,
\]
and therefore \((x,\eta)\in\mathcal S_{\tau_2}\).

Next, let \(\xi\in\mathcal K_{\tau_1}(x,\eta)\). Then, for every
\(i\in\mathcal I_{\max}(x)\),
\[
\nabla g_i(x)^\top \xi
\le
-\gamma\rho(x,\eta)+\alpha(B_{\tau_1}(x,\eta)).
\]
Since \(\alpha\) is strictly increasing and
\(B_{\tau_2}(x,\eta)\ge B_{\tau_1}(x,\eta)\), one has
\[
\alpha(B_{\tau_2}(x,\eta))\ge \alpha(B_{\tau_1}(x,\eta)).
\]
Therefore,
\[
\nabla g_i(x)^\top \xi
\le
-\gamma\rho(x,\eta)+\alpha(B_{\tau_2}(x,\eta)),
\quad
\forall i\in\mathcal I_{\max}(x),
\]
which proves that \(\xi\in\mathcal K_{\tau_2}(x,\eta)\). Hence
\[
\mathcal K_{\tau_1}(x,\eta)\subseteq \mathcal K_{\tau_2}(x,\eta).
\]

By the definition
\[
\mathcal U_\tau(x,\eta)
=
\{u\in\mathbb R^n\mid -\nabla f(x)+u\in\mathcal K_\tau(x,\eta)\},
\]
the inclusion
\[
\mathcal K_{\tau_1}(x,\eta)\subseteq \mathcal K_{\tau_2}(x,\eta)
\]
immediately yields
\[
\mathcal U_{\tau_1}(x,\eta)\subseteq \mathcal U_{\tau_2}(x,\eta).
\]

Assume now that the quadratic programs defining \(u_{\tau_1}(x,\eta)\) and \(u_{\tau_2}(x,\eta)\) are well posed. Since \(u_{\tau_j}(x,\eta)\) is the minimum-norm element of \(\mathcal U_{\tau_j}(x,\eta)\), \(j=1,2\), and since
\[
\mathcal U_{\tau_1}(x,\eta)\subseteq \mathcal U_{\tau_2}(x,\eta),
\]
one obtains
\[
\|u_{\tau_2}(x,\eta)\|
=
\min_{u\in\mathcal U_{\tau_2}(x,\eta)}\|u\|
\le
\min_{u\in\mathcal U_{\tau_1}(x,\eta)}\|u\|
=
\|u_{\tau_1}(x,\eta)\|.
\]

Finally, suppose \(x\in\partial\mathcal C_g\) and \(\eta=0\). Then
\[
\bar g(x)=0, 
\rho(x,0)=0, 
B_{\tau_j}(x,0)=\gamma\tau_j, 
j=1,2.
\]
By definition,
\[
\xi_{\tau_j}=-\nabla f(x)+u_{\tau_j}(x,0)\in \mathcal K_{\tau_j}(x,0),
\quad j=1,2.
\]
Hence, for every \(i\in\mathcal I_{\max}(x)\),
\[
\nabla g_i(x)^\top \xi_{\tau_j}
\le
\alpha(\gamma\tau_j),
\quad j=1,2.
\]
Using Lemma \ref{lem:gbar_directional}, we conclude that
\[
D\bar g(x;\xi_{\tau_j})
=
\max_{i\in\mathcal I_{\max}(x)} \nabla g_i(x)^\top \xi_{\tau_j}
\le
\alpha(\gamma\tau_j),
\quad j=1,2.
\]

Therefore, the admissible upper bound on the outward motion of \(\bar g\) at the feasible boundary becomes weaker as \(\tau\) increases. The final interpretation follows by combining the set inclusions, the minimum-norm comparison, and the boundary estimate above.
\end{proof}

\begin{remark}
Let \(\mathcal C_i\) and \(\mathcal C_j\) be two distinct path-connected components of \(\mathcal C_g\). Define
\begin{equation*}
\begin{aligned}
\tau_{ij}\triangleq
\inf\bigl\{\tau>0 \mid\;&
\mathcal C_i\times\{0\}\ \text{and}\ 
\mathcal C_j\times\{0\}\\
&\text{belong to the same path component of }\mathcal S_\tau
\bigr\}.
\end{aligned}
\end{equation*}
with the convention that \(\tau_{ij}=+\infty\) if the set above is empty.
According to Theorem \ref{thm:rsgf_recurrent} and Theorem \ref{prop:tau_tradeoff}, if a trajectory of \eqref{eq:rsgf_closed_loop} with initial condition \((x(0),\eta(0))\in \mathcal C_i\times\{0\}\) later reaches a state \((x(t),\eta(t))\) such that \(x(t)\in\mathcal C_j\), then necessarily \(\tau\ge \tau_{ij}\). Hence \(\tau_{ij}\) characterizes a purely topological obstruction threshold for transitions between feasible components.

This threshold should be interpreted only as a necessary accessibility condition. Even when \(\tau\ge \tau_{ij}\), the actual closed-loop vector field may still fail to generate such a transition, and no convergence conclusion follows from path-connectivity alone.
\end{remark}

\section{Simulations}
\label{sec: simulations}
In this section, we present two numerical examples to illustrate the qualitative behavior of the fixed horizon timed recovery safe gradient flow and to compare it with both a projected gradient flow and a standard gradient flow.

\begin{example}\label{Example1}
    Consider a two-dimensional inequality-constrained nonlinear program:
    \begin{equation}\label{eq:Example1}
        \begin{aligned}
            \min_{x_1,x_2}\quad & f(x_1,x_2)=(x_1-0.5)^2+x_2^2\\
        \mbox{s.t.}\quad &-\sin(\pi x_1)+x_2\le 0,\\
        &-1\le x_1\le 1,\\
        &-2\le x_2.
        \end{aligned}
    \end{equation}
\end{example}

According to \eqref{eq:Example1}, the feasible set $\mathcal C_g$ is
\[
\mathcal C_g=\{x\in\mathbb R^2\mid g_i(x)\le 0,\ i=1,\dots,4\},
\]
where
\begin{equation*}
\begin{aligned}
    &g_1(x)= -\sin(\pi x_1)+x_2,\\
    &\begin{aligned}
    &g_2(x)=x_1-1, &g_3(x)=-x_1-1,\\
    &g_4(x)=-x_2-2. &
\end{aligned}
\end{aligned}
\end{equation*}
Since \(\nabla g_1(x)=(-\pi\cos(\pi x_1),1)^\top\) is nonzero for all \(x\), and the remaining constraints are affine box-type constraints, a direct case-wise verification shows that Assumption \ref{ass:descent_direction_weak} holds.
The initial condition is chosen as
\[
x(0)=(-0.95,-0.4)^\top,\quad \eta(0)=0,
\]
which is feasible. The fixed parameters are
\[
\gamma=2,\quad \lambda=1000,\quad \alpha(s)=s,
\]
and the simulation horizon is \(T=5\). 
The timed recovery safe gradient flow is simulated for
\[
\tau\in\{0.5,\ 1.5,\ 2.5,\ 3.5\}.
\]
For comparison, we also implement both a projected gradient flow and a standard gradient flow, initialized identically.
For the timed recovery safe gradient flow, the augmented closed-loop dynamics are integrated by a fourth-order Runge--Kutta method. At each step, the control input is generated by the minimum-norm timed recovery controller \eqref{eq:rsgf_qp_u}. 

\begin{figure}[t]
    \centering
        \begin{subfigure}{0.47\columnwidth}
        \includegraphics[width=\linewidth]{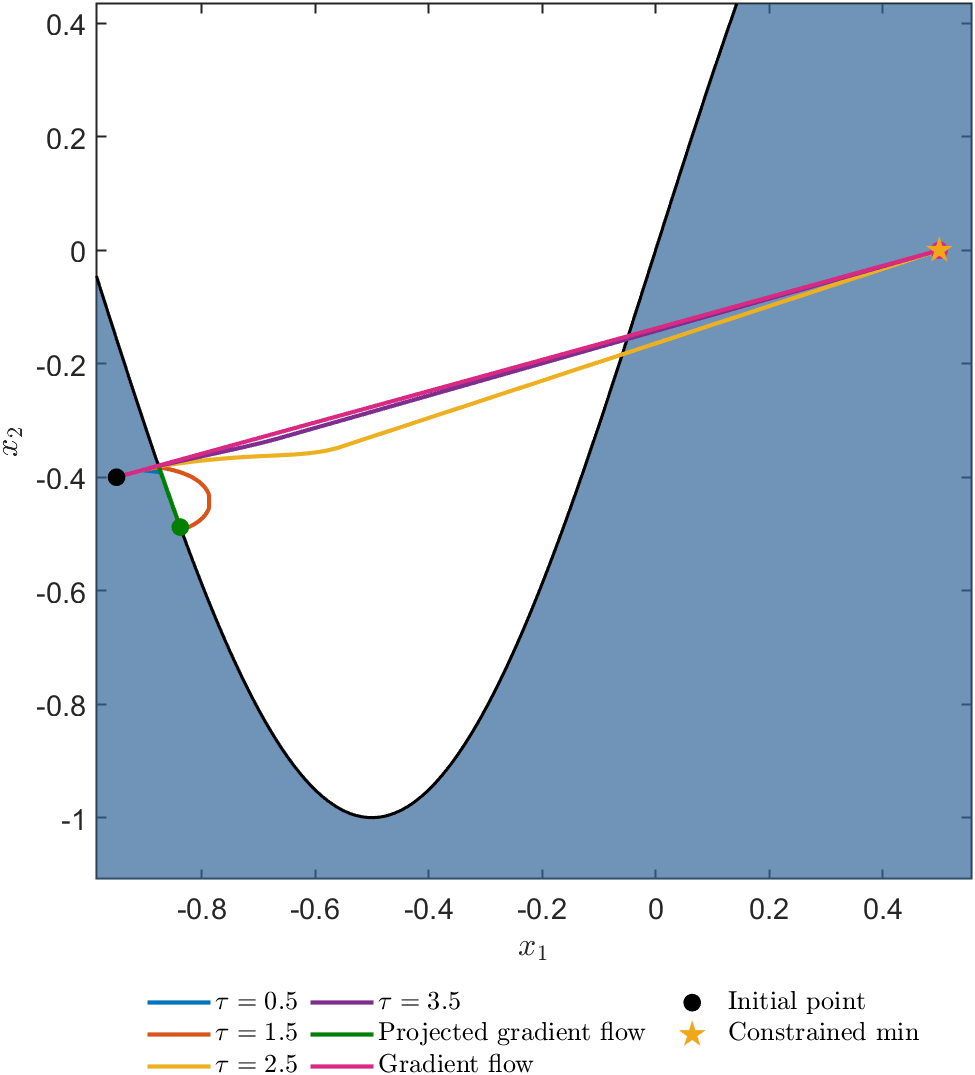}
        \caption{System trajectory}
        \label{fig:simu2_trajectory}
    \end{subfigure}
    \hfill 
    \begin{subfigure}{0.49\columnwidth}
        \includegraphics[width=\linewidth]{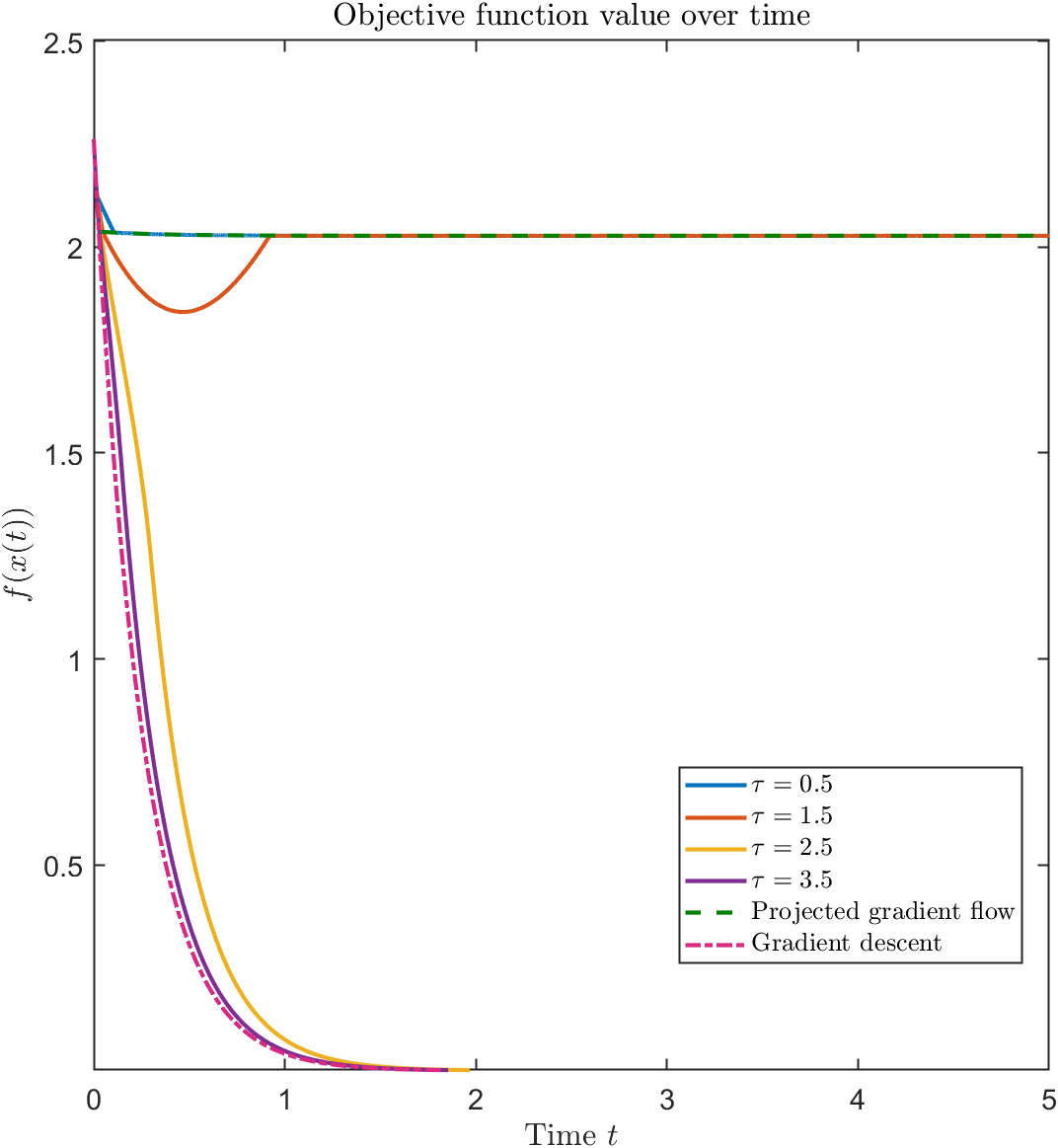}
        \caption{Function value}
        \label{fig:simu2_value}
    \end{subfigure}
    \caption{Trajectories in Example \ref{Example1}}
\end{figure}
Fig.~\ref{fig:simu2_trajectory} shows the trajectories in the \((x_1,x_2)\)-plane.
The black marker denotes the initial point, and the pentagram denotes the unconstrained minimizer, which is also the constrained minimizer,
\[
x^*=(0.5,0)^\top.
\]
Fig.~\ref{fig:simu2_value} illustrates the evolution of the objective function value $f(x)$.
These curves indicate that a larger $\tau$ allows the trajectory to follow the nominal gradient direction more closely and thereby yields a faster descent in this example, which is consistent with Proposition \ref{prop:tau_infty_nominal}.
Moreover, near the optimal point $x^*$, the trajectory converges and the objective value decreases at an exponential rate, in agreement with Theorem \ref{thm:fixed_tau_local_interior_convergence}.

\begin{example}\label{Example2}
    Consider a two-dimensional inequality-constrained nonlinear program  \cite{Gomez1982,sasena2002}:
    \begin{equation}\label{eq:Example2}
        \begin{aligned}
            \min_{x_1,x_2}\quad &f(x_1,x_2)=(4-2.1x_1^2+\frac{x_1^4}{3})x_1^2+x_1x_2+(4x_2^2-4)x_2^2\\
        \mbox{s.t. }\quad &-\sin(4\pi x_1)+2\sin^2(2\pi x_2)\le 0,\\
         &-1\le x_i\le 1,\quad i=1,2.
        \end{aligned}
    \end{equation}
\end{example}

Analogously, according to \eqref{eq:Example2}, the feasible set $\mathcal C_g$ is
\[
\mathcal C_g=\{x\in\mathbb R^2\mid g_i(x)\le 0,\ i=1,\dots,5\},
\]
where
\begin{equation*}
\begin{aligned}
    &g_1(x)= -\sin(4\pi x_1)+2\sin^2(2\pi x_2),\\
    &\begin{aligned}
    &g_2(x)=x_1-1, &g_3(x)=-x_1-1,\\
    &g_4(x)=x_2-1, &g_5(x)=-x_2-1.
\end{aligned}
\end{aligned}
\end{equation*}
The initial condition is chosen as
\[
x(0)=(-0.9,-0.5)^\top,\quad \eta(0)=0,
\]
which is feasible. 
The simulation horizon is set to $T = 20$, while all other parameters remain identical to those in Example \ref{Example1}.

\begin{figure}[t]
    \centering
        \begin{subfigure}{0.48\columnwidth}
        \includegraphics[width=\linewidth]{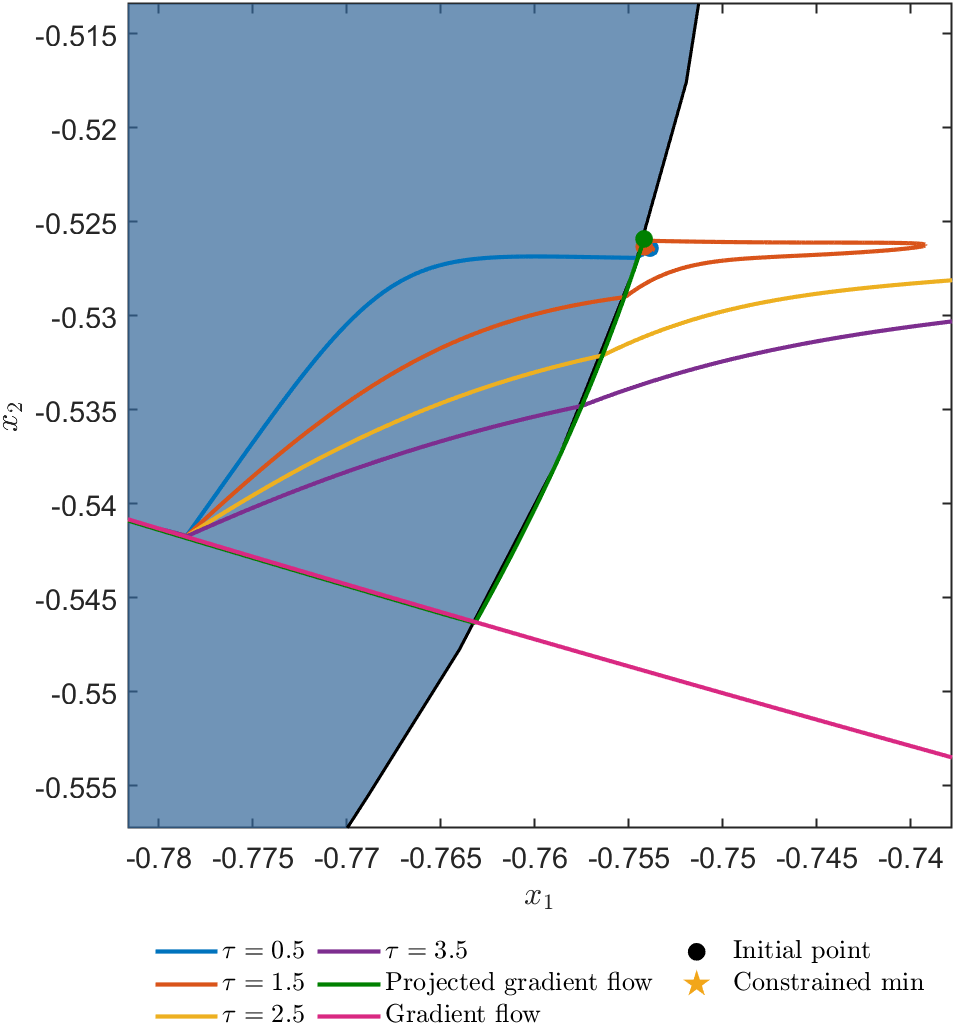}
        \caption{Local trajectory}
        \label{fig:simu1_global_trajectory}
    \end{subfigure}
    \hfill  
    \begin{subfigure}{0.46\columnwidth}
        \includegraphics[width=\linewidth]{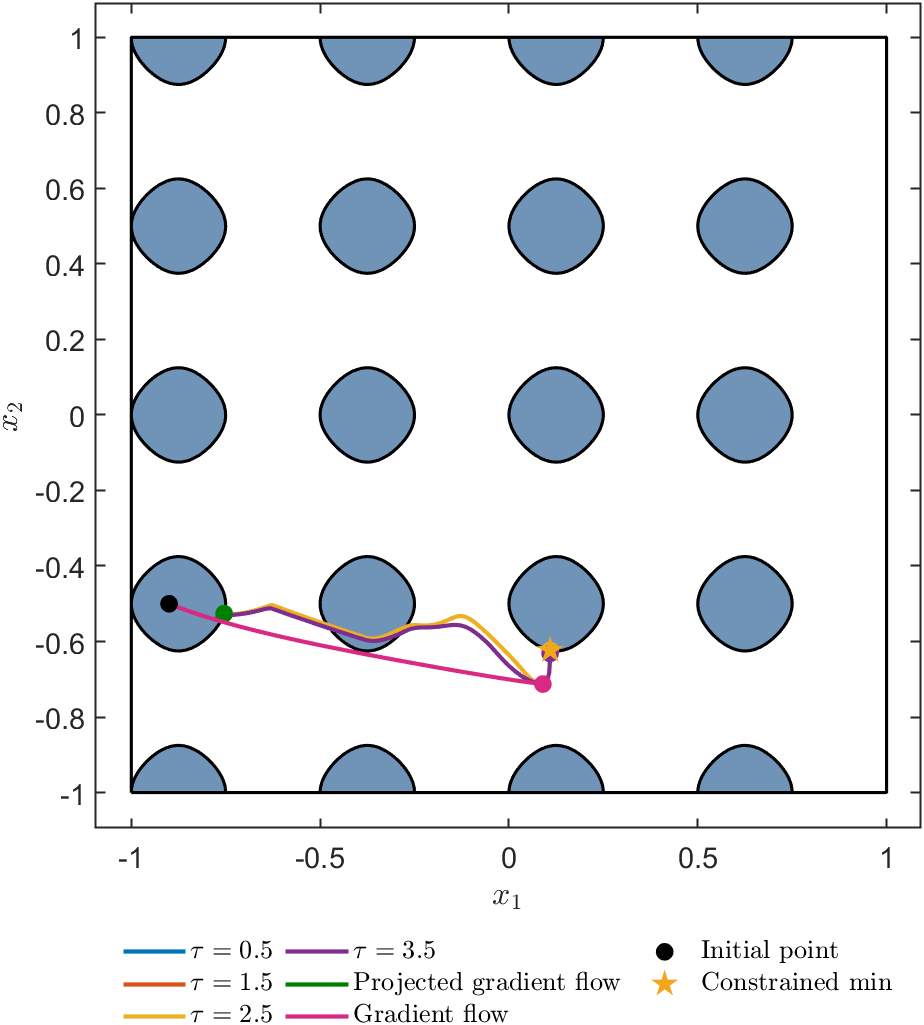}
        \caption{System trajectory}
        \label{fig:simu1_part}
    \end{subfigure}
    \begin{subfigure}{0.47\columnwidth}
        \includegraphics[width=\linewidth]{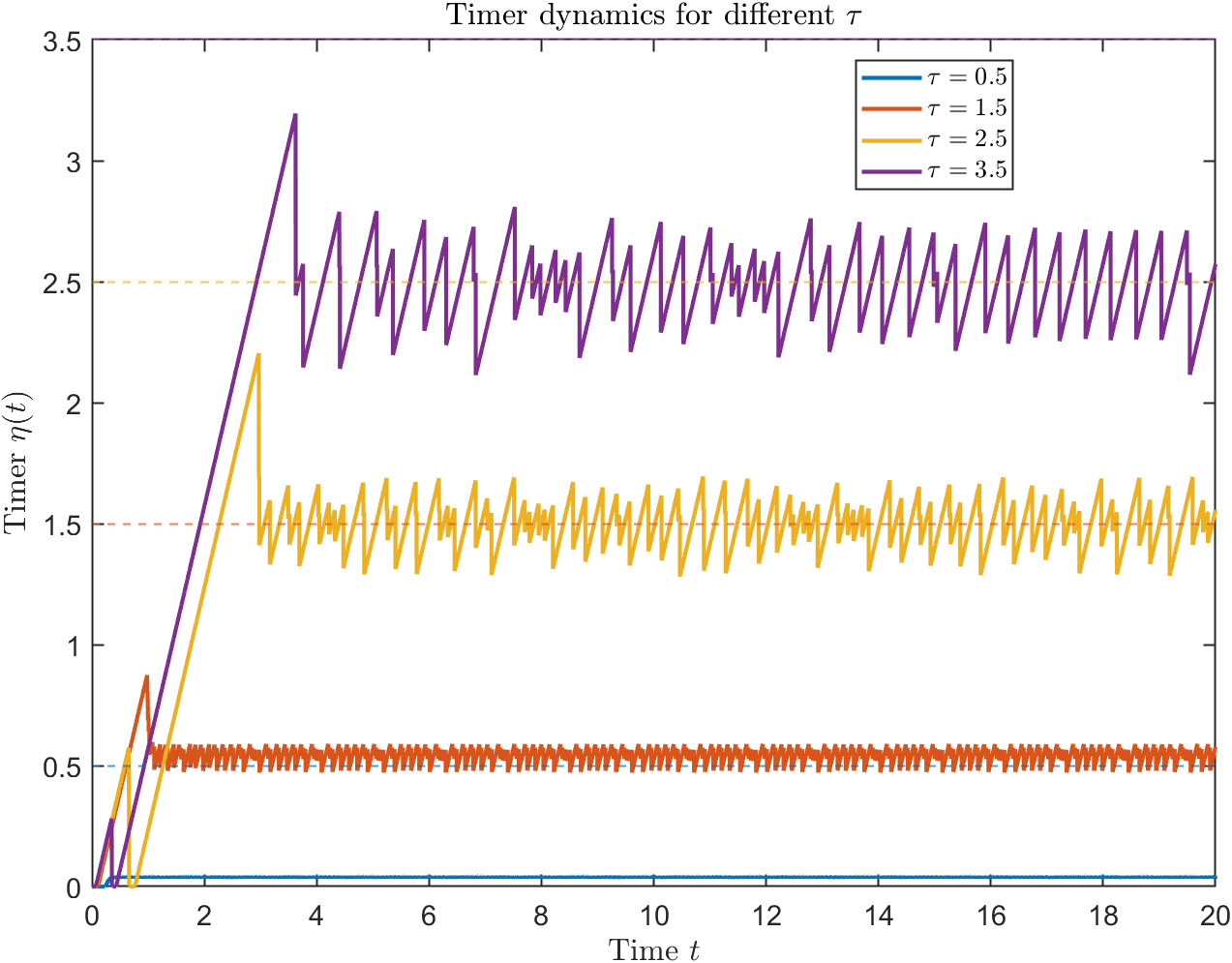}
        \caption{Timer trajectories}
        \label{fig:simu1_timer}
    \end{subfigure}
    \hfill  
    \begin{subfigure}{0.49\columnwidth}
        \includegraphics[width=\linewidth]{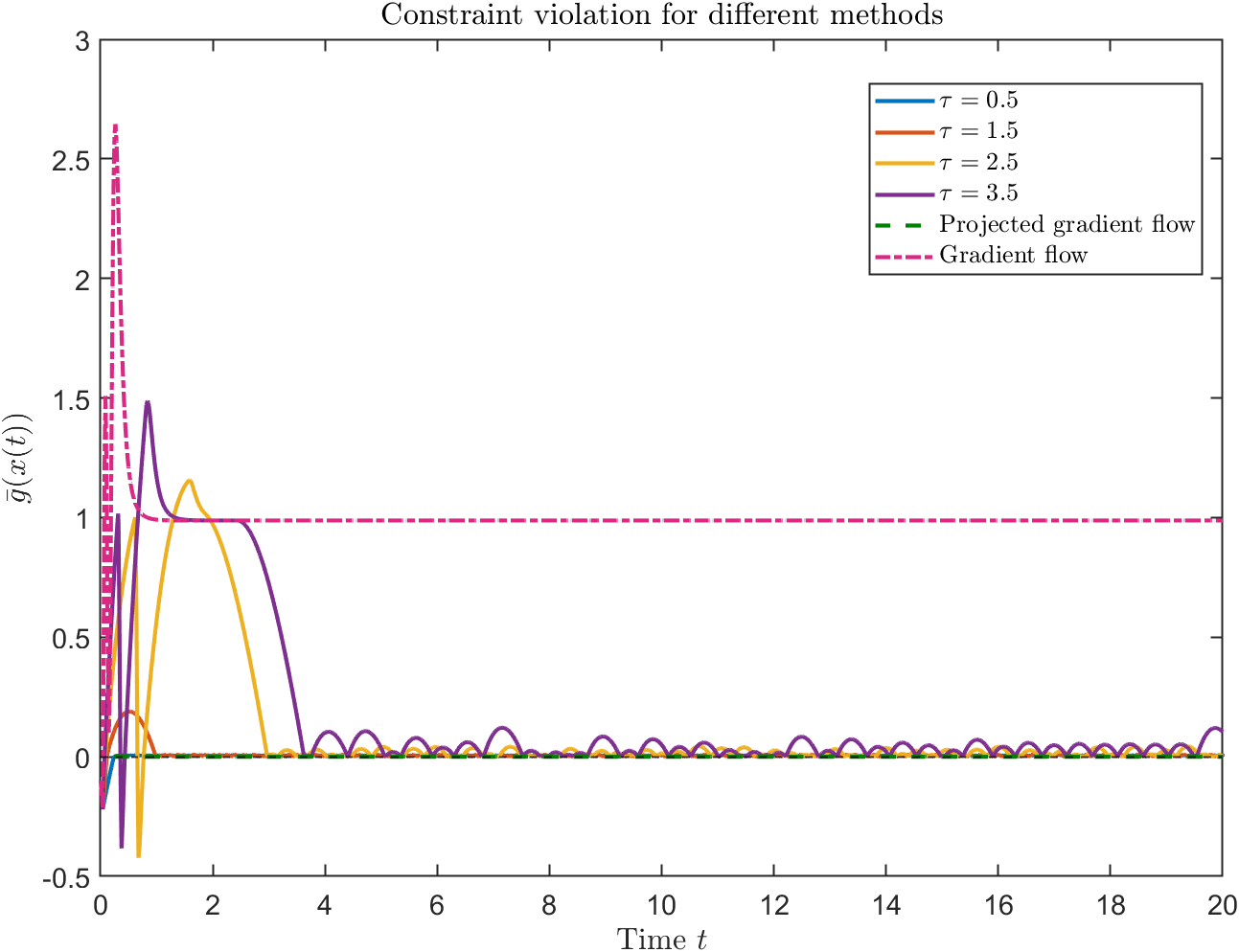}
        \caption{Aggregated violation}
        \label{fig:simu1_violation}
    \end{subfigure}
    \caption{Trajectories in Example \ref{Example2}}
\end{figure}

Figs.~\ref{fig:simu1_global_trajectory}-\ref{fig:simu1_part} show the trajectories in the \((x_1,x_2)\)-plane. The black marker denotes the initial point, and the pentagram denotes the constrained local minimizer
\[
x^\dagger=(0.1093,-0.6234)^\top.
\]
The projected-gradient baseline remains in the feasible region throughout the evolution because each update is followed by a projection onto \(\mathcal C_g\). By contrast, the timed recovery safe gradient flow is not required to preserve pointwise feasibility. Its trajectories may cross the boundary of \(\mathcal C_g\) temporarily, and the extent of this freedom is regulated by the horizon parameter \(\tau\). As \(\tau\) increases, the admissible violation budget
\[
B(x,\eta)=\gamma(\tau-\eta)-\bar g(x)
\]
becomes less restrictive, and the resulting trajectory is allowed to exploit a larger transient infeasible excursion before the restorative mechanism becomes dominant.
When $\tau$ is small ($\tau=0.5$ and $1.5$), the system favors safety; conversely, when 
$\tau$ is large ($\tau=2.5$ and $3.5$), the system favors exploration. This observation is consistent with Theorem \ref{prop:tau_tradeoff}.

Fig.~\ref{fig:simu1_timer} reports the timer trajectories corresponding to the four values of \(\tau\). The horizontal dashed line in each color marks the associated level \(y(t)=\tau(t)\). These curves visualize the internal recovery mechanism of the proposed dynamics. Whenever the trajectory enters the infeasible region, one has \(\bar g(x(t))>0\), and the timer starts to accumulate. Once the state returns to the feasible region, the timer decays according to \(\dot \eta=-1000 \eta\). 

Fig.~\ref{fig:simu1_violation} shows the aggregated violation \(\bar g(x(t))\). For the projected-gradient baseline, the curve stays at or below the zero line because the iterate is projected back to feasibility after each descent step. For the timed recovery safe gradient flow, positive values of \(\bar g(x(t))\) do occur, confirming that the proposed mechanism indeed permits temporary constraint violation. At the same time, these violations are not left unchecked. Once the timer grows and the remaining budget \(\tau-\eta(t)\) shrinks, the correction term becomes increasingly active and drives the trajectory back toward the feasible set. Therefore, the numerical results are fully consistent with the fixed-horizon recurrent interpretation developed in the previous sections: the closed-loop system should in general be understood not as a monotone feasible descent flow, but as a timer-regulated optimization dynamics alternating between descent-oriented motion and bounded-time feasible recovery.
 
Overall, these examples highlight the qualitative difference among these mechanisms. The projected-gradient baseline enforces stepwise feasibility restoration at every update. 

The gradient flow disregards the constraints and directly explores the optimal function value.
The timed recovery safe gradient flow relaxes this requirement and uses the timer state to regulate temporary excursions outside \(\mathcal C_g\). The function value trajectory, phase portrait, timer evolution, and aggregated violation together illustrate how the parameter \(\tau\) controls the tradeoff between conservative feasibility preservation and optimization-oriented exploration. 

\section{Conclusion}
This paper proposed a timed recovery safe gradient flow for inequality constrained continuous-time optimization. Instead of requiring pointwise feasibility at all times, the proposed framework replaces it with fixed horizon feasible recovery. By introducing a timer-regulated recurrent feasibility filter and computing the minimum-norm correction through a pointwise convex quadratic program, the resulting dynamics guarantees forward invariance of the augmented safe set and $\tau$ recurrent return to the original feasible set. The analysis further clarifies the main qualitative features of the proposed mechanism, including its exact agreement with the nominal gradient flow in a strict interior inactive regime and its increasingly restorative behavior as the remaining recovery budget decreases. Crucially, by permitting quantitatively regulated transient excursions outside the feasible boundary, this framework provides a highly promising continuous time approach for optimization problems characterized by disconnected feasible regions, overcoming the topological reachability limitations inherent in strictly forward invariant designs. Future work will study shrinking-horizon recovery mechanisms and the associated feasible reset sequence, with the goal of relating the asymptotic recovery process to the KKT structure of the original problem.

\raggedbottom
\begin{IEEEbiography}[{\includegraphics[width=1in,height=1.25in,clip,keepaspectratio]{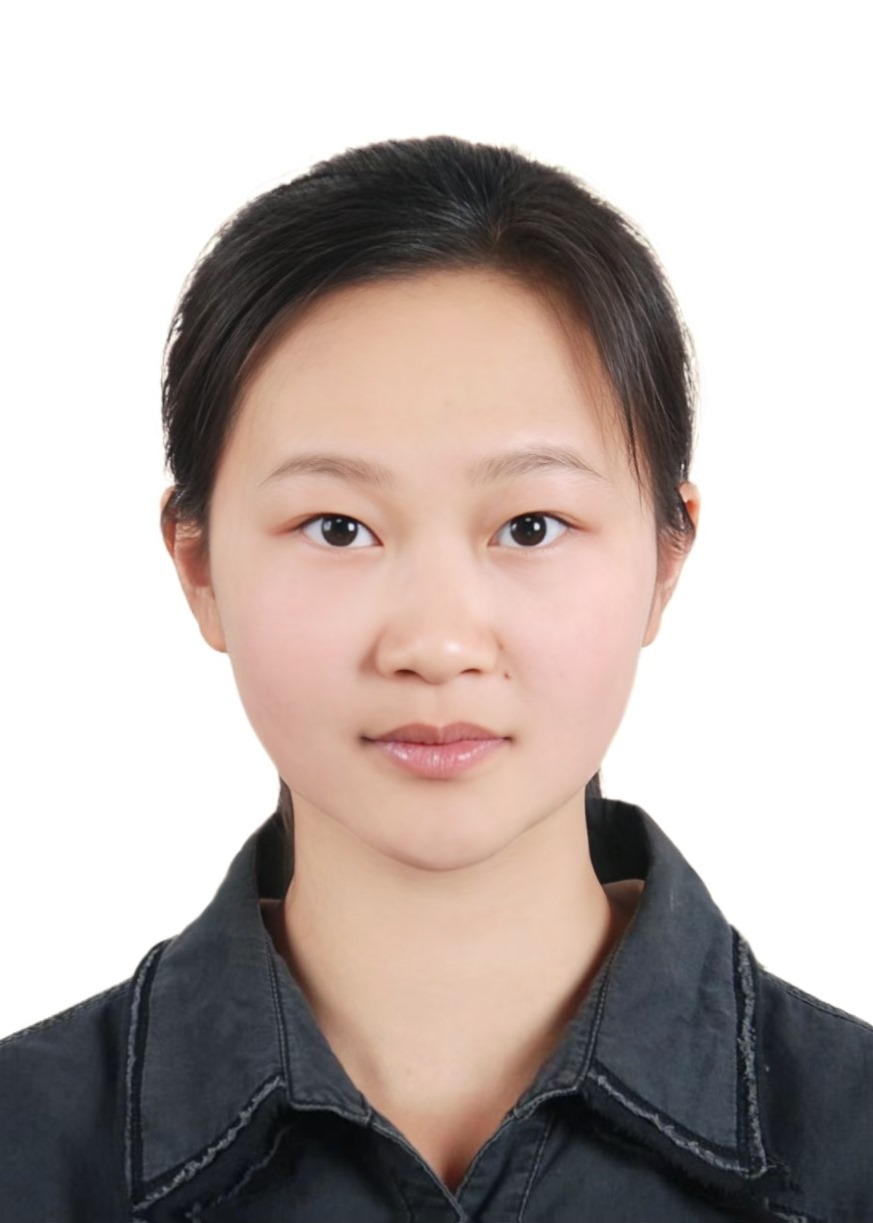}}]{Tianxi Luo}
Tianxi Luo is currently a graduate student with the College of Artificial Intelligence, Nankai University, Tianjin, China.

Her current research interests include control barrier functions, safety-critical control, distributed multi-agent systems, and constrained dynamical systems.
\end{IEEEbiography}

\vspace{-0.6\baselineskip}

\begin{IEEEbiography}[{\includegraphics[width=1in,height=1.25in,clip,keepaspectratio]{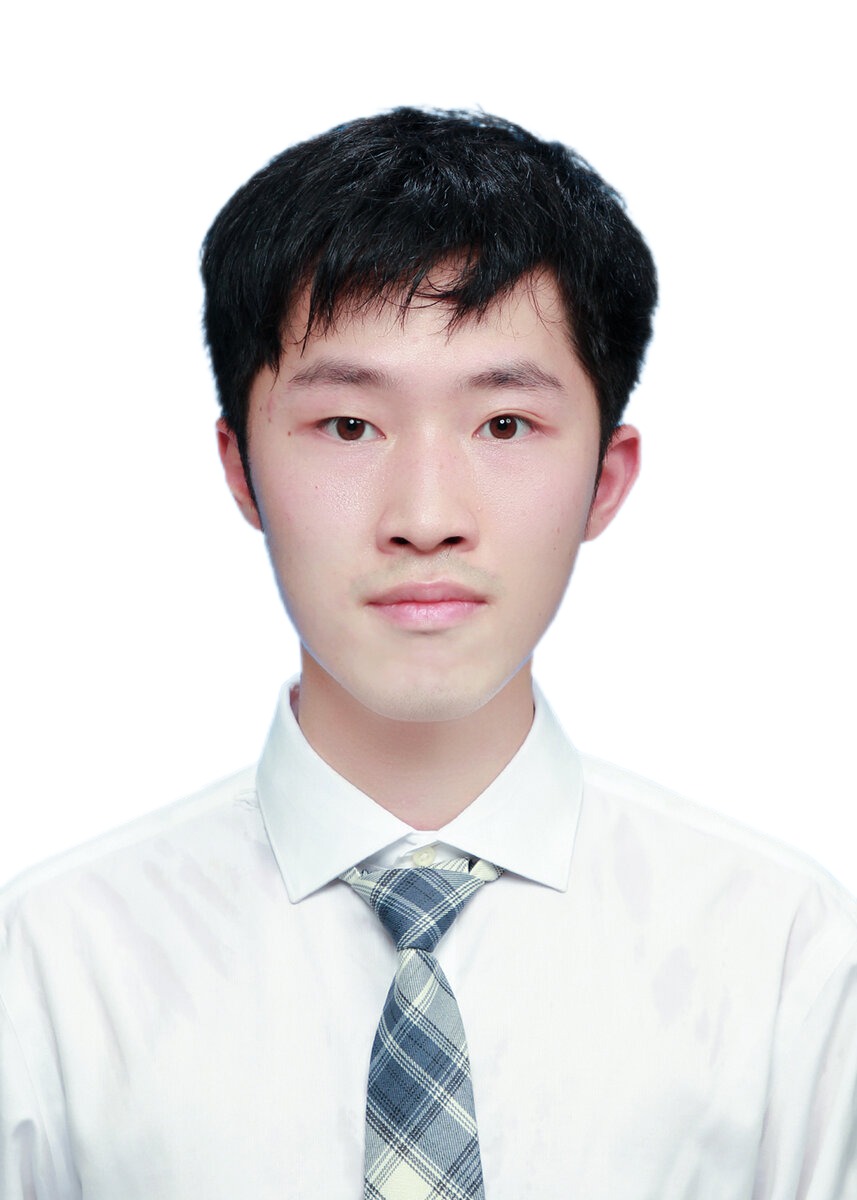}}]{Yue Wu}
Yue Wu is currently a graduate student with the State Key Laboratory of Mathematical Sciences, Academy of Mathematics and Systems Science, Chinese Academy of Sciences, Beijing, China.

His current research interests include safety-critical control, disturbance observer, logical control systems, and game theory.
\end{IEEEbiography}

\vspace{-0.6\baselineskip}

\begin{IEEEbiography}[{\includegraphics[width=1in,height=1.25in,clip,keepaspectratio]{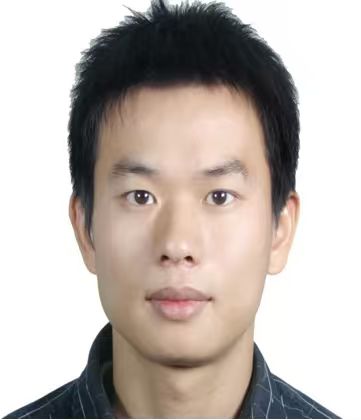}}]{Fei Chen}
(M’12–SM’15) received his Ph.D. in control theory and control engineering from Nankai University, China, in 2009. 

He currently holds the position of Full Professor and serves as the Vice Dean of the College of Artificial Intelligence at Nankai University. Before joining Nankai University, he worked as a postdoc researcher at Utah State University, USA (2009-2010), an Associate Professor at Xiamen University, China, and a Full Professor at Northeastern University, China. Additionally, he held visiting positions at the University of California, USA (2017-2018) and City University of Hong Kong, China (2008, 2013, 2016, 2024). He is the first author of the book “Distributed Average Tracking in Multi-agent Systems” (Springer, 2020). He has published over 100 peer-reviewed papers in international journals and conferences. He was named in the World's Top 2\% Scientist Rankings. He serves as a Senior Editor for IEEE ACCESS. His research interests include distributed control, distributed optimization, and distributed learning.
\end{IEEEbiography}
\vspace{-0.6\baselineskip}

\begin{IEEEbiography}[{\includegraphics[width=1in,height=1.25in,clip,keepaspectratio]{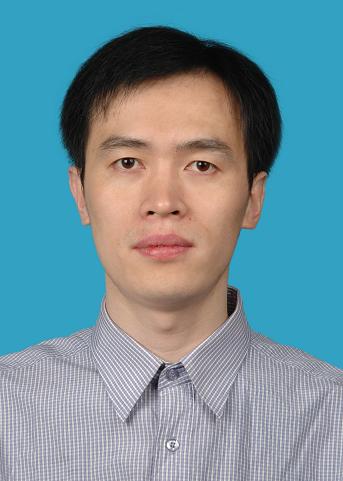}}]{Hongsheng Qi} received his Ph.D. degree in systems theory from Academy of Mathematics and Systems Science, Chinese Academy of
Sciences in 2008. From July 2008 to May 2010, he was a
postdoctoral fellow in the Key Laboratory of Systems Control,
Chinese Academy of Sciences. He currently is a professor with Academy
of Mathematics and Systems Science, Chinese Academy of Sciences. He was a recipient of ``Automatica''
2008-2011 Theory/Methodology Best Paper Prize in 2011 and a recipient of a second National Natural Science Award of China in 2014.
His research interests include logical dynamic systems, game theory, quantum networks, etc.
\end{IEEEbiography} 
\end{document}